\documentclass{amsart}
\usepackage{amsmath}
\usepackage{amsthm}
\usepackage{amsfonts}
\usepackage{amssymb}
\usepackage{enumerate}
\newtheorem{theorem}{Theorem}
\newtheorem{lemma}[theorem]{Lemma}
\newtheorem{prop}[theorem]{Proposition}
\newtheorem{corollary}[theorem]{Corollary}

\theoremstyle{theorem}
\newtheorem{definition}[theorem]{Definition}
\theoremstyle{definition}

\theoremstyle{remark}
\newtheorem{remark}[theorem]{Remark}
\theoremstyle{remark}\newtheorem{claim}{Claim}
\newcommand{\DD}{{\mathbb D}}
\newcommand{\OO}{{\mathcal O}}

\newcommand{\PP}{{\mathbb P}}

\newcommand{\dD}{{\mathcal  D}}
\newcommand{\GG}{{\mathbb G}}

\newcommand{\NN}{{\mathbb N}}

\newcommand{\RR}{{\mathbb R}}

\newcommand{\CC}{{\mathbb C}}

\newcommand{\TT}{{\mathbb T}}
\newcommand{\XX}{{\mathbb X}}
\newcommand{\YY}{{\mathbb Y}}

\newcommand{\la}{\lambda}
\DeclareMathOperator{\Aut}{Aut} \DeclareMathOperator{\id}{id}
 
 \DeclareMathOperator{\re}{Re}

\DeclareMathOperator{\supp}{supp}

\renewcommand{\phi}{\varphi}

\begin{document}

\title{Three-point Nevanlinna Pick problem in~the~polydisc}

\address{Institute of Mathematics, Faculty of Mathematics and Computer Science, Jagiellonian University,  \L ojasiewicza 6, 30-348 Krak\'ow, Poland}
\author{\L ukasz Kosi\'nski}\email{lukasz.kosinski@im.uj.edu.pl}

\thanks{Partially supported by the Ideas Plus grant 0001/ID3/2014/63 of the Polish Ministry of Science and Higher Education}


\maketitle

\begin{abstract}

It is very elementary to observe that functions interpolating an extremal two-point Pick problem on the polydisc are just left inverses to complex geodesics.

In the present article we show that the same property holds for a three-point Pick problem on polydiscs, i.e. it may be expressed it in terms of three-complex geodesics. Using this idea we are able to solve that problem obtaining formulas and a uniqueness theorem for solutions of extremal problems. In particular, we determine a class of rational inner functions interpolating that problem.

Possible extensions and further investigations are also discussed.
\end{abstract}

\section{Introduction}

In this paper we shall be dealing with Pick and Nevanlinna-Pick interpolation problems in the polydisc. The precise definitions and statements are given below. Before we formulate them we shall introduce some preliminary notation: $\DD$ denotes the unit disc in the complex plane, $\TT$ stands for the unit circle, and $H^\infty(\DD^n)$ is the algebra of bounded holomorphic functions in the polydisc $\DD^n$.

A \emph{solvable $N$-Pick interpolation problem} on $\DD^n$ is a set of points $\{z_1,\ldots, z_N\}$ in $\DD^n$ and a set of complex numbers $\{\la_1, \ldots, \la_N\}$ in $\DD$ such that there is an $F$ in the closed unit ball of $H^{\infty}(\DD^n)$ that satisfies $F(z_j) = \la_j$, $j=1,\ldots, N$. Such a function $F$ is called an \emph{interpolating} function. This problem is \emph{extremal} if it is solvable and there are no interpolating functions with norm less than $1$.

The original Pick problem was stated for $n=1$ and solved by Pick in 1916 \cite{Pick} who showed that a necessary and sufficient condition is that the matrix
$$\left( \frac{1- \la_i \bar\la_j}{1- z_i \bar{z}_j }\right)_{i,j=1}^N
$$
is positive semi-definite. Here, all extremal solutions are just Blaschke products of degree at most $N-1$.

It is not clear how to get a solvability criterion in the case of a general domain. An important step towards understanding this problem was achieved by J. Agler \cite{Aglpreprint} who extended in 1988 Pick's theorem to the bidisc using kernel function approach. He showed that the Pick problem is solvable in $\DD^2$ if and only if there are positive semi-definite $N$-by-$N$ matrices $\Gamma$ and $\Delta$ such that 
\begin{equation}\label{eq:introdbidisc}
(1-\la_i \bar \la_j ) = (1- z_i^1 \bar z_j^1) \Gamma_{ij} + (1- z_i^2 \bar z_j^2) \Delta_{ij},
\end{equation} where superscripts are used to denote the coordinates of points $z_j$. The key role in Agler's approach was played by And\^o's inequality which, as is well known (see e.g. \cite{Aglbook}), does not generalize to higher dimensional polydiscs. Consequently, it is impossible to obtain an analogue of criterion \eqref{eq:introdbidisc} for $\DD^n$ if $n$ is greater than or equal to $3$ (see \cite{Aglbook}, Section 11.8).

Following Agler and McCarthy \cite{Aglbook} we shall distinguish Pick and Nevanlinna problems in a somewhat informal way. Roughly speaking the problem whether one can interpolate is called the \emph{Pick problem} and the \emph{Nevanlinna problem} is parametrizing all interpolating functions.
The Nevanlinna problem is difficult to deal with, just to mention it is still open for the bidisc, the best known result is due to Ball and Trent (\cite{Ball-Trent}). It leads to the question on the uniqueness of interpolating functions. It is quite well understood for the two-point Pick problem not only in the polydisc (in fact, it is straightforward there) but also in a bigger class of domains (see \cite{Kos-Zwo 2013}). However, a little is known for a general case of $N$-points. So far, only uniqueness of interpolating functions to a three-point Pick problem in the bidisc has been clear (see Agler and McCarthy \cite{Agl-McNEWYORK}). There are also some partial results due to Guo, Huang and Wang \cite{Guo-Huang-WangD^3} for three-point problem in the tridisc. The uniqueness for the polydisc was also studied in Scheinker's papers \cite{Scheinker}, \cite{Scheinker2}.

The problem of uniqueness is naturally connected with the so-called \emph{uniqueness variety}. Let us recall its definition. Having a solvable $N$-Pick problem on $\DD^n$ denote by $\mathcal V$ the set of points in $\DD^n$ on which all the interpolating functions have the same value. This set is usually a proper subset of $\DD^n$ and the argument of Agler and McCarthy \cite{Agl-McACTA} shows that $\mathcal V$ is a variety, what justifies its name. Agler and McCarthy \cite{Agl-McNEWYORK} described uniqueness varieties for the three-Pick problem in $\DD^2$. In \cite{Agl-McACTA} they showed that any uniqueness variety contains a distinguished variety $\mathcal W$ (that is a variety of the form $\mathcal W=\{(z,w)\in \DD^2:\ p(z,w)=0\}$ for some polynomial $p$ and such that $\overline{\mathcal W} \cap \partial (\DD^2) = \overline{\mathcal W}\cap \TT^2$) that contains each of nodes.

\bigskip

The aim of the present paper is to solve a three-Pick (Nevanlinna) problem in an arbitrary polydisc. Generally speaking we shall show that an extremal problem $\DD^n\to \DD$ may be expressed in terms of the classical Pick-interpolation problem $\DD\to \DD^n$. The relation between these problems is expressed in terms of a $2-1$ real analytic function. The most interesting is that we are able to work with polydiscs in $\CC^n$ even for $n\geq 3$, where kernel-function methods fail. Note however that the results presented here are new even for the bidisc.

Our approach allows us to determine the class of rational inner functions interpolating three Pick problem in $\DD^n$. Surprisingly it turns out that for arbitrary $n$ it is possible to construct such a class composed of functions depending only on three variables. This means that the  three-point Pick problem on $\DD^n$ always boils down to the problem in $\DD^3$.

We shall also be dealing with the problem of uniqueness of interpolating functions. In particular, we shall obtain a distinguished variety contained in the uniqueness variety for that problem. Moreover, we shall show that, briefly saying, the problem is unique if and only if it is $2$-dimensional (see Section~\ref{sec:Statement of results} for a precise definition of a $2$-dimensional problem).

\bigskip

To formulate the main result we need to introduce the notion of $N$-extremals and $N$-complex geodesics which appeared recently as very natural generalizations of both Blaschke products and extremal mappings. They were defined in \cite{Agl-Lyk-You 2013} and \cite{Kos-Zwo 2014}. Precise definitions are postponed to Section~\ref{sec:extr}, as to formulate our result the following characterizations are sufficient: a \emph{$3$-complex geodesic} in $\DD^n$ is a mapping $f:\DD\to \DD^n$ with all components being a Blaschke product of degree at most 2. Additionally, $F:\DD^n\to \DD$ is its \emph{left inverse}, if $F\circ f$ is a Blaschke product of degree at most $2$. One can see that demanding that Blaschke products occurring above are of degree $1$ we obtain definitions of standard extremals and their left inverses.

\bigskip

Now we are in a good position to state our main result in the most general form:
\begin{theorem}\label{main}
Any function interpolating a non-degenerate extremal $3$-point Pick problem on $\DD^n$ is a left inverse to a $3$-complex geodesic passing through the nodes.
\end{theorem}
Lemmas \ref{main:bidisc}, \ref{main:tridisc} and \ref{main:poli} contain more detailed versions of this result.

\bigskip

We shall also show (see Section~\ref{sec:GTFP}) that there is no hope to generalize Theorem~\ref{main} to the case of $N$-points, $N\geq 4$. This in particular means that the case of three points is very special. 

Some connections with problems arising in Geometry Function Theory, especially with the Lempert theorem, are also discussed in the paper.

\section{Statement of results}\label{sec:Statement of results}
Let $z_1,z_2, z_3\in \DD^n$ be pairwise distinct points and let $\sigma_1,\sigma_2, \sigma_3\in \DD$ be distinct. We say that the Pick-interpolation problem $\DD^n\to \DD$ mapping each $z_j$ to $\sigma_j$ is \emph{$k$-dimensional} if it can be interpolated by a function depending only on $k$-variables. It is said to be \emph{strictly $k$-dimensional} if additionally it is not $k-1$-dimensional.

If some of two point subproblem mapping the pair $(z_i,z_j)$ to $(\lambda_i,\lambda_j)$ is extremal, we shall call the problem \emph{degenerate}. Otherwise it is called \emph{non-degenerate}.

Since the polydisc is homogenous, considering three-point Pick interpolation problem we may always restrict ourselves to the following situation
$$(*)
\begin{cases}
0\mapsto 0\\
z\mapsto \sigma\\
w\mapsto \tau,
\end{cases}
$$
where $z,w\in \DD^n\setminus \{0\}$, $z\neq w$ and $\sigma,\tau\in \DD$.

Let us start with the result for the bidisc.
\begin{lemma}\label{main:bidisc}
Let $n=2$. Suppose that the problem (*) is extremal, non-degenerate, and strictly $2$-dimensional.

Then there are $\alpha\in \DD^2$, $\alpha_1\neq \alpha_2$ and $\eta\in \TT$ such that
 $\lambda\mapsto (\lambda m_{\alpha_1}(\lambda), \eta \lambda m_{\alpha_2}(\lambda))$ passes through the nodes and 
$$F((\lambda m_{\alpha_1}(\lambda), \eta \lambda m_{\alpha_2}(\lambda)) = \lambda m_{\gamma}(\lambda),\quad \la\in \DD,
$$ where $\gamma = t\alpha_1 + (1-t)\alpha_2$ for some $t\in (0,1)$.

Moreover, there is a unique interpolating function of the form \begin{equation}\label{form:F} F(z_1,z_2)=\frac{t z_1 + \bar \eta (1-t) z_2 + \omega \bar \eta z_1 z_2}{1 + ((1-t) z_1 + t \bar \eta z_2)\omega},\quad z\in \DD^2,
\end{equation} 
where $\omega = \frac{\bar\alpha_1 - \bar \alpha_2}{\alpha_1 - \alpha_2}$.

If the problem is 1-dimensional, extremal and non-degenerate then, up to a permutation of components, 
\begin{multline}\nonumber
|z_1|<|w_1|, \ |z_2|<|w_2| \text{ and } \rho (\frac{z_1}{w_1}, \frac{z_2}{w_2}) \leq \rho (w_1, w_2) \quad \text{or}\\ z_2=\omega z_1, \ w_2=\omega w_1 \text{ for some } \omega\in \TT.
\end{multline} 
In other words, there is $\phi\in \OO(\DD,\overline \DD)$ such that, up to a permutation of components, $\lambda\mapsto (\lambda, \lambda \varphi(\lambda))$ passes through the nodes. Here, if $\phi\not\in \TT\cup \Aut(\DD)$, the interpolating function is unique and depends only on the first variable.

\end{lemma}

An analogous result remains true for the tridisc:
\begin{lemma}\label{main:tridisc}
Let $n=3$. Let (*) be extremal, non-degenerate, and strictly $3$-dimensional.
Then there is an interpolating function 
\begin{equation}\label{eq:maintridisc}
F(z_1,z_2, z_3)  = F_1(F_2(z_1,z_2),z_3),\quad z\in \DD^3,
\end{equation} where $F_1$, $F_2$ are rational functions of form \eqref{form:F}.

Additionally there are $\omega\in \TT^{2}$ and $\alpha\in \DD^3,$ where $\alpha_1, \alpha_2, \alpha_3$ are not co-linear, such that an analytic disc $$\lambda \mapsto (\lambda m_{\alpha_1}(\lambda), \omega_1 \lambda m_{\alpha_2}(\lambda),  \omega_2 \lambda m_{\alpha_3}(\lambda))$$ passes through the nodes and $$F(\lambda m_{\alpha_1}(\lambda), \omega_1 \lambda m_{\alpha_2}(\lambda),  \omega_2 \lambda m_{\alpha_3}(\lambda)) = \lambda m_\gamma (\lambda),\quad \la\in \DD,$$ where $\gamma$ is a convex combination of $\alpha_1, \alpha_2, \alpha_3$.  In this case all interpolating functions coincide on a 2-dimensional analytic variety and the function that interpolates is never unique.

Moreover, if the problem is non-degenerate and strictly 2-dimensional, then, up to a permutation of components, there are $\varphi\in \mathcal O(\DD, \DD)$, $\alpha\in \DD^2$, $\omega\in \TT$ such that the 3-extremal $\lambda\mapsto (\lambda m_{\alpha_1} (\lambda), \omega \lambda m_{\alpha_2}(\lambda), \la \phi(\lambda))$ passes through the nodes and $$F(\lambda m_{\alpha_1} (\lambda), \omega \lambda m_{\alpha_2}(\lambda), \lambda \phi(\lambda))= \lambda m_\gamma (\lambda),\quad \la\in \DD,$$ where $\gamma$ is a convex combination of $\alpha_1$ and $\alpha_2$. Here generically the interpolating function is unique (e.g. if $\varphi$ is not a M\"obius map).
In particular, the problem may be directly reduced to the bidisc.

If the problem is non-degenerate and 1-dimensional, then, up to a permutation of variables, $|z_j|<|z_1|,$ $|w_j|<|w_1|,$ $\rho(\frac{z_j}{z_1},\frac{w_j}{w_1})< \rho(z_1, w_2)$ or $z_j= \omega z_1$, $w_j=\omega w_1$ for some $\omega\in \TT$, $j=1,2$. Thus there are $\varphi_1,\varphi_2\in \mathcal O(\DD, \overline \DD)$ such that, up to a permutation of components, $\lambda\mapsto (\lambda, \lambda \phi_1(\lambda), \lambda \phi_2(\lambda))$ passes through the nodes. The interpolating function is unique and depends on one variable if both $\phi_1$ and $\phi_2$ are neither unimodular constants nor automorphisms of the unit disc.
\end{lemma}

\begin{remark}
As we shall see there is a formula for $\alpha$, $\omega$, and points at which the three extremal meets nodes. The formula is given in an implicit way as an inverse to a rational 2-fold function (which is additionally even with respect to the nodes and $\alpha$).
\end{remark}

\begin{lemma}\label{main:poli}
Let $n\geq 4$. Then any interpolating function of extremal non-degenerate problem (*) is a left inverse to a 3-complex geodesic passing through the nodes. Moreover, this interpolation problem is solvable if and only if it is solvable by an inner rational function being, up to a permutation of variables, of the form \eqref{eq:maintridisc}.

In particular, the Pick problem on $\DD^n$ may be reduced directly to $\DD^3$.
\end{lemma}

As already mentioned the uniqueness variety for the extremal Pick problem on $\DD^n$ is always a proper subset of $\DD^n$ if $n\geq 3$, provided that the problem is not $2$-dimensional. However, we are able to determine the variety on which all interpolating functions do agree. Here we shall state the result for $\DD^3$. A general case is explained in Lemma~\ref{lem:convexcombination}. The results presented here means, in particular, that uniqueness proven in \cite{Guo-Huang-WangD^3} for some cases is very special, as it may occur only for $2$-dimensional problems.
\begin{prop}\label{thm:uniq}
Let $n=3$. Let (*) be an extremal, maximal and non-degenerate and let $F$ be its left inverse. Then $F$ is uniquely determined on the $2$-dimensional variety.

More precisely, the equality
\begin{equation}\label{eq:leftF} F(\lambda m_{t\alpha_1}(\lambda), \omega_1 \lambda m_{t\alpha_2}(\lambda),  \omega_2 \lambda m_{t\alpha_3}(\lambda)) = \lambda m_{t\gamma} (\lambda),
\end{equation} holds for any $\lambda\in \DD$ and $t\in [0,1]$.

Moreover, the mapping $F$ satisfying equality~\eqref{eq:leftF} is never unique.
\end{prop}

\begin{remark}The idea used in Proposition~\ref{thm:uniq} (see Lemma~\ref{lem:convexcombination} for details) applied to the bidisc gives another proof of uniqueness of functions interpolating a three-Pick non-degenerate and strictly 2-dimensional problem on $\DD^2$ (see \cite{Agl-McNEWYORK}).
\end{remark}

The paper is organized as follows. We start with some quite elementary preparatory results. Next we focus on left inverses. In particular, special class of rational functions lying in the Schur class on the polydisc is constructed there and a converse to that construction is proven. Then we introduce the crucial notion of extremal mappings and complex geodesics. The core of the paper is hidden in Section~\ref{sec:main}.

Further sections present the proof of the main theorem. We start with the case of the bidisc, which turns out to be a quite straightforward consequence of ideas developed in Section~\ref{sec:main}. That case is used in the proof of the main theorem for the tridisc. Similarly, to get the case of arbitrary polydisc we shall involve both results: for the bidisc as weel as for the tridisc. 

Finally, in Section~\ref{sec:GTFP} we shall present some connections and applications to geometric function theory problems and discuss possible generalizations.

\section{Preliminaries}
Thanks to the transitivity of the group of automorphisms of the polydisc we restricted ourselves to the following problem
$$
(*)
\begin{cases}
0\mapsto 0\\
z\mapsto \sigma\\
w\mapsto \tau,
\end{cases}
$$ where $z\neq w,$ $z\neq 0$, $w\neq 0$ and $(\sigma,\tau)\neq (0,0)$. Let us denote 
\begin{equation}\label{eq:dn} 
\dD_n:=\{(z,w)\in \DD^n\times \DD^n:\ z\neq w, z\neq 0, w\neq 0\}.
\end{equation}

A standard Montel-type argument shows that for any $(z,w)\in \dD_n$ and any $(\sigma,\tau)\in \CC^2\setminus \{(0,0)\}$ there is exactly one $t=t_{z,w,(\sigma, \tau)}>0$ such that the problem $\DD^n\to \DD$: $0\mapsto 0$, $z\mapsto t\sigma$ and $w\mapsto t\tau$ is extremally solvable.
It is simple to see that the mapping
$$\dD_n \times \CC^2\setminus\{(0,0)\}\ni(z,w,(\sigma,\tau))\mapsto t_{z,w,(\sigma,\tau)} \in \RR_{>0}$$
is continuous.

Moreover, for fixed nodes $z,w$ the mapping $$ [\sigma:\tau]\mapsto (t_{z,w,(\sigma,\tau)} \sigma, t_{z,w,(\sigma,\tau)}\tau)$$ gives a 1-1 correspondence between the projective plane $\PP_1$ and the set of target points for which $(*)$ is extremally solvable modulo a unimodular constant.

Saying about the extremal three-point Pick problem (shortly the extremal problem) corresponding to the data $(z,w,[\sigma:\tau])$ we mean:
$$
\begin{cases}
0\mapsto 0\\
z\mapsto t_{z,w,(\sigma,\tau)}\sigma\\
w\mapsto t_{z,w,(\sigma,\tau)} \tau.
\end{cases}
$$ In particular, the target points $t_{z,w,(\sigma,\tau)} \sigma$ and $t_{z,w,(\sigma,\tau)} \tau$ are determined up to a unimodular constant.

\bigskip

Here is some additional notation. $\Delta_n=\{t\in \RR^n:\ t_j\geq 0, \sum t_i=1\}$ stands for the closed, $n$-dimensional simplex. Let $\Delta_n^\circ$ and $b\Delta_n$ denote its interior and boundary respectively, i.e. $\Delta_n^\circ=\{t\in \RR^n: t_j>0, \sum t_j=1\}$ and $b\Delta_n = \Delta_n\setminus \Delta_n^\circ$. Idempotent M\"obius maps are denoted by $m_\alpha$, where $\alpha\in \DD$, i.e. $m_\alpha(\la)=\frac{\alpha - \la}{1 - \bar \alpha \la},$ $\la\in \DD$. $t\cdot \alpha$ stands for a standard dot product of $t$ and $\alpha$.
The hyperbolic distance in the unit disc is denoted by $\rho$.

We use $\mathcal H^p$ for $p$-dimensional Hausdorff measure. Recall that for any domain $D$ in $\CC^n$ and a closed subset $A$ such that $\mathcal H^{n-1}(A)=0$ the set $D\setminus A$ is a domain. We shall use this fact in the sequel several times.

\subsection{Left inverses}\label{leftinverses}
Let $\alpha=(\alpha_1, \ldots, \alpha_n)\in \DD^n$ and $t=(t_1,\ldots, t_n)\in \Delta_n.$ We shall inductively construct a function $F=F_{\alpha,t}:\DD^n\to \DD$ such that
\begin{equation}\label{eq:left}
F(\lambda m_{\alpha_1}(\lambda), \ldots, \lambda m_{\alpha_n}(\lambda)) = \lambda m_{\gamma}(\lambda),\quad \lambda \in \DD,
\end{equation} 
where $\gamma = t\cdot \alpha$.

The key role in the construction is played by the so-called \emph{magic functions}:
\begin{equation}\label{eq:magic} \Phi_{s,\eta}(z_1,z_2)=\frac{sz_1 + (1-s) z_2  - \eta z_1 z_2}{1- ((1-s) z_1 + s z_2) \eta},
\end{equation}
where $\eta\in \TT$ and $s\in [0,1]$. Note that $\Phi_{s,\eta}:\DD^2\to \DD$. 

\begin{lemma}\label{lem:propmagic}
Let $s\in (0,1)$ and $\eta\in \TT$. Let $\phi, \psi\in \mathcal O(\DD, \DD)$. Assume that
$$
\lambda \mapsto \Phi_{s,\eta}(\la \phi(\la), \la \psi(\la))
$$
is a Blaschke product of degree $2$. Then both $\phi$ and $\psi$ are M\"obius maps.

If additionally $\phi(\la) = m_{\alpha_1}(\la)$ and $\psi(\la) =\omega m_{\alpha_2}(\la),$ $\la\in \DD$, where $\alpha_1,\alpha_2\in \DD$, $\alpha_1\neq \alpha_2$, and $\omega\in \TT$, then $\omega=1$ and $\eta = \frac{\alpha_1 - \alpha_2}{\bar \alpha_1 - \bar \alpha_2}$. 

Moreover, for such $\eta$ the following relation holds:
$$
\Phi_{s,\eta}(\lambda m_{\alpha_1}(\lambda), \lambda m_{\alpha_2}(\lambda)) =\lambda m_{s\alpha_1 +(1-s)\alpha_2}(\lambda),\quad \lambda\in \DD.
$$
\end{lemma}

\begin{proof}
Let us denote $\gamma(\lambda) = \Phi_{s\eta}(\la \phi(\la), \la \psi(\la))/\la,$ $\la\in \DD$. Note that
\begin{multline}\nonumber
|\gamma'(0)| = |s \phi'(0) + (1-s) \psi'(0) + \eta s (1-s) (\phi(0) - \psi(0))^2|\leq\\
1 - |s \phi(0) + (1-s)\psi(0)|^2 = 1- |\gamma(0)|^2.
\end{multline}
Thus the assertion is a consequence of the Schwartz lemma.
\end{proof}

For $n=2$ we put $$F_{\alpha, t}:= \Phi_{t,\eta}, \text{ where } \eta=\frac{\alpha_1 - \alpha_2}{\overline{\alpha_1} - \overline{\alpha_2}}.$$

Let $\alpha\in \DD^n,$ $t\in \Delta_n$. Having constructed a desired left inverse $F'=F_{\alpha',t'}:\DD^{n-1}\to \DD$ satisfying \eqref{eq:left} for $\alpha'=(\alpha_1,\ldots, \alpha_{n-1})\in \DD^{n-1}$ and $t'$ in a $n-1$ dimensional simplex, we may inductively define a left inverse $\DD^n\to \DD$. To do it for $\alpha\in \DD^n$ and $t\in \Delta_n$ we put
\begin{equation}\label{def:leftF} F(z',z_n)= F_{\alpha, t} (z', z_n) := \Phi_{s,\eta}(F'(z'), z_n),
\end{equation}
where $s\in [0,1]$, $\eta\in \TT$ and $t'$ are chosen so that the function $F$ defined in this way satisfies \eqref{eq:left}.

\begin{remark}
If $n\geq 4$ the definition of the left inverse $F$ presented above is not unique. On the other hand, if $\alpha_1$, $\alpha_2$, $\alpha_3$ are not co-linear, any point lying in a convex combination od $\alpha_1,\alpha_2$ and $\alpha_3$, uniquely defines its barycentric coordinates, whence $F_{\alpha, \omega, t}$ is uniquely defined in the way presented above for $n=2$ or $n=3$.
\end{remark}

To define a left inverse to a general $3$-complex geodesic, formula \eqref{def:leftF} may be naturally extended in the following way:
\begin{definition}\label{def:leftFgen} Let $\alpha\in \DD^n$, $t\in \Delta_n$, $\omega\in \TT^{n-1}$. Put:
$$F_{\alpha, \omega, t}(z) = F_{\alpha, t}(z_1, \bar \omega_1 z_2, \ldots, \bar \omega_{n-1} z_n),\quad z\in \DD^n.$$
\end{definition}
As already mentioned, formulas $F_{\alpha, \omega, t}$ have sense for an $n$, but they are defined in a unique way only for $n=2,3$. We shall need the following technical observation:
\begin{lemma}\label{claim:omegaunique}
If $n=2$ (resp. $n=3$) and $\alpha_1\neq \alpha_2$ (resp. $\alpha_1$, $\alpha_2$, $\alpha_3$ are not co-linear) left inverses $F_{\alpha,\omega,t}$, where $t\in \Delta_n^\circ$, have the following property: if the function $$\lambda\mapsto F_{\alpha, \omega, t}(\lambda m_{\beta_1}(\lambda), \eta_1 \lambda m_{\beta_2}(\lambda),\ldots, \eta_n\lambda m_{\beta_n}(\lambda))$$ is a Blaschke product of degree $2$, then $\omega_i=\eta_i$ for $i=1,\ldots, n-1$, and $\frac{\alpha_1 - \alpha_2}{\bar \alpha_1 - \bar \alpha_2}=\frac{\beta_1 - \beta_2}{\bar \beta_1 - \bar \beta_2}$.
\end{lemma}

\begin{proof}
For $n=2$ this is an immediate consequence of Lemma~\ref{lem:propmagic}. If $n\geq 3$ it suffices to use directly the definition of $F_{\alpha, \omega, t}$ together with Lemma~\ref{lem:propmagic}.
\end{proof}

The subsequent results may viewed as inverse to that construction.
\begin{lemma}\label{lem:convexcombination}
Let $\alpha_1,\ldots, \alpha_n\in \DD$. Assume that $F:\DD^n\to \DD$ is a holomorphic function such that $F(\lambda m_{\alpha_1}(\lambda), \ldots, \lambda m_{\alpha_n}(\lambda)) = \lambda m_\gamma(\lambda)$ for $\la\in \DD$, where $\gamma\in \DD$. Then $\gamma$ is a convex combination of $\alpha_1,\ldots, \alpha_n$.

Moreover, for any $t\in [0,1]$ the following equality holds:
\begin{equation}
F(\lambda m_{t\alpha_1}(\lambda), \ldots, \lambda m_{t\alpha_n}(\lambda)) = \lambda m_{t\gamma}(\lambda),\quad \lambda\in \DD.
\end{equation}
\end{lemma}

\begin{proof}
For $0\leq s< s_0:=\min |\alpha^{-1}_j|$ put $$g_s(\lambda):= F(\lambda m_{s\alpha_1}(\lambda), \ldots, \lambda m_{s\alpha_n}(\lambda))/\lambda,\quad \lambda\in \DD.$$

Clearly, $g_s$ is a self-mapping of the unit disc, so it follows from the Schwartz lemma that $|g_s(0)|^2 + |g_s'(0)|\leq 1$ for any any $s\in [0,s_0)$. Moreover, the assumption implies that the equality here is attained for $s=1$, even more the equality $g_1'(0)= -1 +|g_1(0)|^2$ holds. Rewriting the above inequality we get 
$$|\sum_j F_{z_j}'(0)  s \alpha_j|^2+ |\sum_j F_{z_j}'(0)(-1+ |s\alpha_j|^2) + \frac12 \sum_{i,j} F_{z_i, z_j}''(0) s^2 \alpha_i \alpha_j|\leq 1.$$ In particular, the inequality
$$
s^2\big( |\sum_j F_{z_j}'(0) \alpha_j|^2 - \re (\sum_j F_{z_j}'(0)|\alpha_j|^2 - \frac 12 \sum_{i,j} F_{z_i z_j}''(0) \alpha_j \alpha_j)\big) + \re (\sum_j F_{z_j}'(0)) \leq 1
$$
holds for $s\in [0,s_0]$
with the equality attained for $s=1$. This means the left hand side of the inequality above is constant viewed as a function of $s$. Thus putting $s=0$ we find that $\re(\sum_j F_{z_j}'(0))=1$. On the other hand $\sum_j |F_{z_j}'(0)|\leq 1,$ as $F$ maps the polydisc into the disc. This means that $t_j:=F_{z_j}'(0)\geq 0$ and consequently $t=(t_1,\ldots, t_n)\in \Delta$. Clearly, $\gamma=g_1(0) = \sum_j F_{z_j}'(0) \alpha_j = t\cdot \alpha$.

Note that we have shown that $g_s'(0)= -1 + |g_s(0)|^2$ for any $s\in [0,s_0)$. Thus the second part of the assertion is a consequence of computations that have been already carried out.
\end{proof}

Using this result it is very elementary to deduce the uniqueness of left inverses in the bidisc.

\begin{prop}\label{prop: 2, uniqueness}
Let $F:\DD^2\to \DD$ be such that $F(\lambda m_{\alpha_1}(\lambda), \lambda m_{\alpha_2}(\lambda))=\lambda m_\gamma (\lambda)$, $\la\in \DD$, where $\alpha_1\neq \alpha_2$, $\gamma=t\alpha_1 + (1-t) \alpha_2$, and $t\in (0,1)$. Then $F$ is uniquely determined (and thus it is of the form $F_{t,\alpha}$).
\end{prop}

\begin{proof}
The result is a consequence of the proof of Lemma~\ref{lem:convexcombination}. Precisely, we have shown that for any $s\in [0,1]$ the equality $$F(\lambda m_{s\alpha_1}(\lambda), \lambda m_{s\alpha_2}(\lambda)) = \lambda m_{s\gamma}(\lambda),\quad \lambda \in \DD,$$ holds. Let us complexify the mappings $t\mapsto m_{t\alpha_j}(\lambda)$ in a neighborhood of $0$ to holomorphic functions $t\mapsto m_j(t,\lambda)$, $j=1,2$. It is quite direct to observe that the set of points $(\lambda m_1(t ,\lambda),\lambda m_2(t,\lambda))$, where $t$ and $\lambda$ are small, covers an open, non-empty subset of the bidisc. Thus, making use of the identity principle one gets the assertion.
\end{proof}

\begin{remark}
Observe that the uniqueness presented above does not remain true for $n\geq 3$. More precisely, there are many left inverses to $\lambda\mapsto (\lambda m_{\alpha_1}(\lambda),\ldots, \lambda m_{\alpha_n}(\lambda))$ for any $n\geq 3$. 

Note that it suffices to show the lack of uniqueness only for $n=3$. 
Actually, having two different left inverses to analytic disc $\phi: \lambda\mapsto (\lambda m_{\alpha_1}(\lambda), \lambda m_{\alpha_2} (\lambda), \lambda m_{\alpha_3}(\lambda))$ and repeating the construction presented in Section~\ref{fs} one may easily define two different left inverses to $\phi$.

If $n=3$ we may construct left inverses in two ways: either as we did it, that is considering $\Phi_{t,\omega}(\Phi_{s,\eta}(z_1,z_2), z_3)$, $z\in \DD^3$, or taking $\Phi_{t',\omega'} (z_1, \Phi_{s',\eta'}(z_2, z_3))$, $z\in \DD^3$, with suitable chosen $t,s,t',s'\in (0,1)$ and $\omega, \omega', \eta, \eta'\in \TT$. Observe that these left inverses are not identically equal providing that $\alpha_1,\alpha_2$ and $\alpha_3$ are not co-linear. Actually, to avoid tedious calculations one may compare coefficients of the term $z_1z_2z_3$ to get that the equality $\Phi_{t,\omega}(\Phi_{s,\eta}(z_1,z_2), z_3)\equiv \Phi_{t',\omega'} (z_1, \Phi_{s',\eta'}(z_2, z_3))$ on $\DD^3$ would imply that $\omega \eta = \omega' \eta'$. It is not difficult to see that this relation may be satisfied only if $\alpha_1$, $\alpha_2,$ $\alpha_3$ are co-linear.

On the other hand the argument presented in the proof of Proposition~\ref{prop: 2, uniqueness} shows that all functions  interpolating a non-degenerate extremal problem in $\DD^3$ coincide on a 2 dimensional analytic variety of a very special form.
\end{remark}

\section{N-extremal mappings; definition and basic properties}\label{sec:extr}

In the present section we shall recall the notion of $3$-extremal mappings in domains od $\CC^n$. Definitions presented here occurred recently in \cite{Kos-Zwo 2014} and \cite{Agl-Lyk-You 2013}.

Assume that $D\subset \mathbb C^n$ is a domain and $N\geq 2$. Fix pairwise distinct points $\lambda_1,\ldots,\lambda_N\in \mathbb D$ and points $z_1,\ldots,z_N\in D$. As usual, the interpolation data
\begin{equation}
 \lambda_j\mapsto z_j,\ \mathbb D\to D,\quad j=1,\ldots,m,
\end{equation}
is said to be {\it extremally solvable} if it is solvable i.e. there is a map $h\in\OO(\mathbb D,D)$ such that $h(\lambda_j)=z_j$, $j=1,\ldots,N$,
and there is no $f$ holomorphic on a neighborhood of $\overline{\mathbb D}$ with the image in $D$ such that $f(\lambda_j)=z_j$, $j=1,\ldots,N$.

Note that the latter condition is equivalent to the fact that there is no $h\in \mathcal O(\DD, D)$ such that $h(\lambda_j)= z_j$, $j=1,\ldots, N$, and $h(\DD)$ is relatively compact in $D$.

\begin{definition}\label{def:extr}(see \cite{Kos-Zwo 2014}).
Let $f:\DD\to D$ be an analytic disc. Let $\lambda_1,\ldots,\lambda_N\in \DD$ be distinct points. We say that $f$ is a {\it weak $N$-extremal with respect to $\lambda_1,\ldots,\lambda_N$} if the problem 
\begin{equation}\label{eq:aglmextr}
\lambda_j\mapsto f(\lambda_j)
\end{equation} 
is extremally solvable.

Naturally, we shall say that $f$ is a {\it weak $N$-extremal} if it is a weak extremal with respect to some $N$ distinct points in the unit disc.
\end{definition}

The idea of the above definition has roots in \cite{Agl-Lyk-You 2013}, where authors introduced the notion of extremal maps, demanding that problem \eqref{eq:aglmextr} is extremal for all choices of pairwise distinct points $\la_1,\ldots, \la_N$. Definition~\ref{def:extr} is strictly weaker and corresponds to standard extremals introduced by Lempert for $N=2$. However, for many domains classes of $N$-extremals and weak $N$-extremals coincide (see \cite{Kos-Zwo 2014}). This is the case for among others homogenous (i.e. with transitive group of holomorphic automorphisms) and balanced domains. In particular, both definitions are equivalent for the polydisc.
\begin{definition}(see \cite{Kos-Zwo 2014}).
Let $N\geq 2$. An analytic disc $f:\DD\to D$ is called a complex $N$-geodesic if there is a holomorphic function $F:D\to \DD$ such that $F\circ f$ is a Blaschke product of degree $N-1$.

The function $F$ is called a left inverse to the complex $N$-geodesic $f$.
\end{definition}

Any complex $N$-geodesic is $N$-extremal. The following self-evident observation may be viewed as a starting point of our considerations:
\begin{lemma}\label{lem:extr}
Let $f:\DD \to \DD^n$ be a complex $N$-geodesic. Let $F:\DD^n\to \DD$ be its left inverse. Let $b$ denote the Blaschke product $F\circ f$. Then for any pairwise distinct $\la_1,\ldots, \lambda_N\in \DD$ the $N$-point Pick interpolation problem in $\DD^n$:
$$f(\lambda_j)\mapsto b(\la_j),\quad j=1,\ldots, N,$$ is extremal.
\end{lemma}

\begin{proof}
The result is a direct consequence of the extremality of Blaschke products.
\end{proof}

Finally recall that all classes of analytic discs listed above coincide in the polydisc. Moreover, thanks to argument involving Schur's reduction, they may be described in a very simple way: an analytic disc $f:\DD\to \DD^n$ is $N$-extremal if and only if $f_j$ is a Blaschke product of degree at most $N-1$ for some $j=1,\ldots, n$.

\section{Non-degenerate and maximal dimensional problem for the bidisc and tridisc}

Having a function $F:\DD^n\to\DD$ interpolating an extremal three-point Pick problem we are looking for a 3-extremal $\phi$ passing through the given nodes and such that $F$ is its left inverse.

\begin{remark}\label{rem:degreeb}
It is simply to see that if a component $\varphi_j$ is not a Blaschke product of degree $2$, then the relation $F\circ \phi =b$, where $b$ is a Blaschke product of degree at most $2$, implies that $F$ does not depend on $j$-variable.
\end{remark}

Roughly speaking, we shall find an (open) set of points such that the interpolation problem is extremal, non-degenerate, and of maximal dimension, where the latter term means shortly that the problem is strictly $n$-dimensional. It will be relatively simple to check that in the case of the bidisc the set constructed in this way is connected. The situation is more complicated in the tridisc. To overcome this difficulty we shall exploit some technical approach.

It is interesting on its own whether the set of points for which the interpolating function is extremal, non degenerate, and of maximal dimension is in general connected. The positive answer to this problem would simplify some parts of the proof.

\bigskip

For an $n$-fold permutation $\sigma$ and a point $z$ in $\CC^n$ denote $z_\sigma=(z_{\sigma(1)},\ldots, z_{\sigma(n)})$. Moreover, let $\sigma(z,w,\xi)= (z_\sigma, w_\sigma,\xi)$ for $(z,w,\xi)\in \CC^n\times \CC^n\times \PP_1$.

Let $\mathcal A_d$ denote the set of points in $\dD_n\times \PP_1$, where $\dD_n$ is given by \eqref{eq:dn}, such that the extremal Pick interpolation problem corresponding to $(z,w,\xi)$ is degenerate. Note that 
\begin{equation}\label{eq:A_dsubs}
\mathcal A_d\subset \mathcal A_0\cup \mathcal A_d'\cup \mathcal A_d'',
\end{equation} 
where 
\begin{align*} \mathcal A'_0=\{(z,w,\xi):\  \xi &= [z_i:w_i] \text{ for some } i\},\\ 
\mathcal A_d'= \mathcal A_z' \cup \mathcal A_w'=\{& (z,w,\xi):\  |z_i|=|z_j| \text{ and } \xi\neq [0:1]\ \text{ for some } i\neq j\} 
\\
\cup
\{ (z,w,&\xi):\   |w_i|=|w_j| \text{ and }\xi\neq [1:0] \text{ for some } i\neq j\},
\text{ and}\\
\mathcal A_d'' = \{(z,w,\xi):\  \rho&(z_i,z_j)=\rho(w_i,w_j)\text{ and } \xi\neq[1:1] \text{ for some } i\neq j\}.
\end{align*}
It is quite straightforward to see that $\mathcal A_d$ is closed in $\dD_n\times \PP_1$. This fact together with inclusion \eqref{eq:A_dsubs} gives:

\begin{lemma}\label{lem:nondeg-connected} For any subdomain $G$ of $\dD_n$ the set $(G\times \PP_1)\setminus \mathcal A_d$ is a domain. In particular, $(\dD_n\times \PP_1)\setminus \mathcal A_d$ is a domain.
\end{lemma}

\begin{proof}
Note that $\mathcal A:=\mathcal A_d\cap (\mathcal A_d'\cup \mathcal A_d'')$ is closed in $\dD_n$. Since $\mathcal A'_0$ is analytic, it suffices to show the connectedness of $(G\times \PP_1)\setminus \mathcal A$. We shall show that it is path-connected.

Fix two points $(z',w',\xi')$ and $(z'', w'',\xi'')$ in $(G\times \PP_1)\setminus \mathcal A$. Losing no generality, we may assume that none of equalities $|z_i'|=|z_j'|$, $|w_i'|=|w_j'|$, $|z_i''|=|z_j''|$, $|w_i''|=|w_j''|$ is satisfied for any $i\neq j$.

Connect $(z',w')$ and $(z'',w'')$ with a curve $\gamma:[0,1]\to G$ so that $\gamma$ does not touch any intersection of two different sets of the form $\{(z,w)\in \dD_n:\ |z_i|=|z_j|\}$, $\{(z,w)\in \dD_n:\ |w_i|=|w_j|\}$ and $\{(z,w)\in \dD_n:\ \rho (z_i, z_j)= \rho(w_i,w_j),\}$, where $i\neq j$. This is possible as the Hausdorff dimension of any such intersection is equal to $2n-2$.

This construction ensures us that none of points $(\gamma(t), \xi)$, where $t\in [0,1]$ and $\xi \in \PP_1$, lies simultaneously in some two sets of the form $\mathcal A_z'$, $\mathcal A_w'$ or $\mathcal A_d''$.

Making use of the inclusion $\mathcal A\subset \mathcal A_d'\cup \mathcal A_d''$ one can find a curve $\xi:[0,1]\to \PP_1$ such that $t\mapsto (\gamma(t), \xi(t))$ connects $(z',w',\xi')$ and $(z'', w'', \xi'')$ in $(\dD_n\times \PP_1)\setminus \mathcal A$.
\end{proof}

\begin{remark}
Note that  using formulas for left inverses in Lemmas~\ref{main:bidisc}, \ref{main:tridisc} and \ref{main:poli} it is quite straightforward to describe explicitly the set $\mathcal A_d\subset \dD_n\times \PP_1$.
\end{remark}
\subsection{Bidisc}

If $z,w\in \dD_2$ are such that either $|z_2|<|z_1|$, $|w_2|<|w_1|,$ and $\rho(\frac{z_2}{z_1}, \frac{w_2}{w_1})\leq \rho(z_1, w_1)$, or $z_2=\omega z_1$, $w_2=\omega w_1$ for some unimodular constant $\omega$, then there is an analytic disc of the form $\lambda \mapsto (\lambda, \varphi(\lambda)),$ where $\phi\in \mathcal O(\DD, \DD)$, passing through the nodes $0,z,w$. In particular, for any $\xi\in \PP_1$ there is a function interpolating the extremal problem corresponding to the data $(z,w,\xi)$ that depends only on one variable, i.e. the problem is 1-dimensional.

Denote $U'_1=\{(z,w)\in \dD_2:\ |z_2|<|z_1|,\ |w_2|<|w_1|,\ \rho (\frac{z_2}{z_1}, \frac{w_2}{w_1})\leq \rho(z_1, w_1)\}$, $U''_1:=\sigma(U'_1)$, where $\sigma$ is a 2-fold permutation permuting $1$ and $2$. Let $\mathcal I:=\{(z,w)\in \dD_2:\ z_2 = \omega z_1, w_2= \omega w_1\text{ for some } \omega \in \TT\}$. Put 
$$U_1:= U'_1 \cup U''_1\cup  \mathcal I.$$ 
Observe that $U_1$ is closed in $\dD_2$. 
\begin{lemma}\label{cl:D-Udomain}
$\dD_2 \setminus U_1$ is a domain.
\end{lemma} 
\begin{proof} To see the connectedness fix any $t,s\in \DD\setminus \{0\}$. Note that the point $((t,t),(s,-s))$ lies in $\dD_2\setminus U_1$. We have to show that any point $(z,w)\in \dD_2\setminus U_1$ may be connected with $((t,t),(s,-s))$ by a curve in $\dD_2\setminus U_1$. 

Up to a permutation one of the following conditions holds: either $|z_1|\geq |z_2|$, $|w_1|\leq |w_2|$ or $|z_1|<|z_2|$, $|w_1|<|w_2|$, $\rho(\frac{z_1}{z_2}, \frac{w_1}{w_2})> \rho(z_2, w_2)$. If the point $(z,w)$ satisfies the first condition, the claim is trivial. If it satisfies the second one, it suffices to note that the set $\{(z,w)\in \DD^2\times \DD^2:\ |z_1|<|z_2|,\ |w_1|<|w_2|,\ \rho(\frac{z_1}{z_2}, \frac{w_1}{w_2})> \rho(z_2, w_2)\}\}$ is connected and that the point $((t,t),(s,-s))$ lies in its closure.
\end{proof}

Let us define $G = \dD_2\setminus U_1$ and
$$\Omega=\Omega(\DD^2):=(G\times \PP_1)\setminus \mathcal A_d.$$

\begin{lemma}\label{lem:Omega(D2)domain}
$\Omega(\DD^2)$ is a domain.
\end{lemma}
\begin{proof}
The result is an immediate consequence of Lemmas~\ref{lem:nondeg-connected} and \ref{cl:D-Udomain}.
\end{proof}

\subsection{Tridisc}

The idea here is to distinguish one and two dimensional data lying on $3$-complex geodesic.

\begin{enumerate}[i)]

\item Let $V_1'=\{(z,w)\in \DD^3\times \DD^3:\ |z_j|<|z_1|,\ |w_j|<|w_1|,\ \rho(\frac{z_j}{z_1},\frac{w_j}{w_1})<  \rho(z_1, w_1),\ j=2,3\}$.
Put 
$$U_1=\bigcup_{\sigma\in \Sigma_3} \sigma(V_1' \times \PP_1),
$$
where $\Sigma_3$ comprises any three $3$-fold permutations switching $j$-th element and $3$, $j=1,2,3$.

\item For $\phi\in \mathcal O(\DD, \DD)$, $\alpha\in \DD^2$, $\omega\in \TT$, and $t\in [0,1]$ set $$\psi_{\alpha,\omega, \phi}(\lambda) = (\lambda m_{\alpha_1}(\lambda), \omega \lambda m_{\alpha_2}(\lambda), \lambda \phi(\lambda)),\quad \la\in \DD,$$ 
and $\gamma_{t,\alpha} = t\alpha_1 + (1-t)\alpha_2$. Let us define
\begin{multline}\nonumber U_2':=\{(\psi_{\alpha, \omega, \phi}(x), \psi_{\alpha, \omega, \phi}(y), [x m_{\gamma_{t,\alpha}}(x): y m_{\gamma_{t,\alpha}}(y)]):\ \alpha_1\neq \alpha_2,\\ x,y\in \DD\setminus \{0\}, x\neq y,\ t\in (0,1),\ \phi\in\OO(\DD, \DD)\setminus \Aut(\DD)\}.
\end{multline} 
and
$$
U_2=\bigcup_{\sigma\in \Sigma_3} \sigma(U_2').$$

\end{enumerate}

Finally, put 
\begin{equation}\label{Omega(D_3)} 
\Omega(\DD^3):=(\dD_3 \times  \PP_1) \setminus (\overline{U_1}\cup \overline{U_2}\cup \mathcal{A}_d).
\end{equation}

\section{Main construction}\label{sec:main}

The main idea of the paper is hidden in the present section. Recall that $\Delta=\Delta_n$ denote the closed simplex. Let $\Delta^\circ = \Delta_n^\circ$ and $b\Delta=b\Delta_n$.

For $n\in \NN$ consider the mapping $$\Phi:\DD\times \DD \times \DD^n \times \Delta \times \TT^{n-1}\to \DD^n\times \DD^n \times \mathbb P_1$$ given by the formula
$$
\Phi(x,y,\alpha,t,\omega)=(\psi_{\omega,\alpha}(x), \psi_{\omega, \alpha}(y), [x m_{t\cdot \alpha}(x): y m_{t\cdot \alpha}(y)]),$$
where 
$$
\psi_{\omega,\alpha}(\lambda) = (\lambda m_{\alpha_1}(\lambda), \omega_1 \lambda m_{\alpha_2}(\lambda),\ldots, \omega_{n-1} \lambda m_{\alpha_n}(\lambda)),\quad \lambda\in \DD.
$$

Let $X(\DD^n):=\Phi^{-1}(\Omega(\DD^n))$. For the simplicity we shall sometimes write $X=X(\DD^n)$ and $\Omega = \Omega(\DD^n)$. 

Observe that $X(\DD^2)$ consists of points $(x,y, \alpha, t, \omega)\in \DD\times \DD \times \DD^2 \times [0,1]\times \TT$ such that $x\neq y$, $x\neq 0,$ $y\neq 0$, $t\in (0,1)$ and $\alpha_1 \neq \alpha_2$.

Similarly, $X(\DD^3)$ comprises $(x,y,\alpha, t, \omega)\in \DD \times \DD\times \DD^3 \times \Delta \times \TT$ such that $x\neq y$, $x\neq 0$, $y\neq 0$, $t\in \Delta^0,$ and points $\alpha_1, \alpha_2, \alpha_3$ are not co-linear (note that an extremal Pick problem corresponding to data in $(\dD_3\times \PP_1)\setminus \Omega(\DD^3)$ may be interpolated either by a function depending on one variable or, up to a permutation of variables in $\CC^3$, a function depending on two variables of the form \eqref{eq:magic}; then use Lemma~\ref{lem:convexcombination}).

\begin{lemma}\label{lem:phixopen} Let $n=2$ or $n=3$. Then
$\Phi(X(\DD^n))$ is open in $\Omega(\DD^n)$.
\end{lemma}

\begin{proof}

We shall prove the theorem simultaneously for $n=2$ and $n=3$. The crucial observation is that $\dim X= \dim \Omega= 4n+2$. This will allow us to make use of the invariance of domain theorem.

\begin{claim}\label{claim}
$\Phi$ is locally injective on $X$. Moreover, it is $2-1$ and even with respect to $x,y, \alpha$ i.e. $\Phi(x,y, \alpha, t, \omega) = \Phi (-x, -y, -\alpha, t, \omega)$  for $(x,y,\alpha, t, \omega)\in X(\DD^n)$.
\end{claim}

\begin{proof}[Proof of claim]
Let $(x,y,\alpha, t, \omega), (x', y', \alpha', t', \omega')\in X$ be such that
\begin{equation}
\label{row}
\Phi(x,y,\alpha, t, \omega) = \Phi(x', y', \alpha', t', \omega'). \end{equation}
As already mentioned, the definition of $X$ guarantees that, in particular, $t,t'\in \Delta^\circ$, $xy\neq 0$, $x\neq y$, and that $\alpha_1\neq \alpha_2$. 
Let $F_{\alpha,\omega,t}:\DD^n\to\DD$ be a left inverse to $\psi_{\omega, \alpha}$, constructed in Section~\ref{leftinverses}.
Since both problems
$$
\begin{cases}
0\mapsto 0,\\ \psi_{\omega, \alpha}(x) \mapsto x m_{t\cdot \alpha} (x)\\
\psi_{\omega, \alpha}(y) \mapsto y m_{t\cdot \alpha} (y)
\end{cases}\quad \text{and}\quad
\begin{cases}
0\mapsto 0,\\ \psi_{\omega', \alpha'}(x') \mapsto x' m_{t'\cdot  \alpha'} (x')\\
\psi_{\omega', \alpha'}(y') \mapsto y' m_{t'\cdot \alpha'} (y')
\end{cases}
$$
are extremal (see Lemma~\ref{lem:extr}), we find from \eqref{row} that $x m_{t\cdot \alpha} (x)=\tau x' m_{t'\cdot \alpha'} (x')$ and $y m_{t\cdot \alpha} (y)=\tau y' m_{t'\cdot \alpha'} (y')$ for some unimodular $\tau$. From this equalities we deduce that $\tau=1$.
Moreover, they imply that $F_{\alpha, \omega, t}\circ \psi_{\omega', \alpha'}$ fixes the origin and sends $x'$ to $ x' m_{t'\cdot \alpha'}(x')$ and $y'$ to $ y' m_{t'\cdot \alpha'}(y')$. Consequently, $F_{\alpha, \omega, t}\circ \psi_{\omega', \alpha'}$ is a Blaschke product of degree $2$. Thus, according to Lemma~\ref{claim:omegaunique}, $\omega = \omega'$ and $\frac{\alpha_1 - \alpha_2}{\bar \alpha_1 - \bar\alpha_2}= \frac{\alpha'_1 - \alpha'_2}{\bar\alpha'_1 - \bar\alpha'_2}.$ 

Composing the equality $(x m_{\alpha_1}(x), x m_{\alpha_2}(x))= (x' m_{\alpha'_1}(x'), x' m_{\alpha'_2}(x'))$ with the function $\Phi_{s,\eta}$ of the form \eqref{eq:magic}, where $s\in [0,1]$ and $\eta=\frac{\alpha_1 - \alpha_2} {\bar\alpha_1 - \bar \alpha_2}$, we get 
\begin{equation}\label{eq:compPhi1}
x m_{s\alpha_1 +(1-s)\alpha_2} (x) = x' m_{s\alpha'_1 +(1-s)\alpha'_2} (x').
\end{equation} 
Analogously, the relation 
\begin{equation}\label{eq:compPhi2}
y m_{s\alpha_1 +(1-s)\alpha_2} (y) = y' m_{s\alpha'_1 +(1-s)\alpha'_2} (y')
\end{equation}
holds for any $s\in [0,1]$. Looking at equations \eqref{eq:compPhi1} and \eqref{eq:compPhi2} as at rational functions depending on $s$ and comparing their coefficients we find that
$$\frac{\alpha_2 - x}{\alpha_1 - \alpha_2}=\frac{\alpha'_2 - x'}{\alpha'_1 - \alpha'_2},\quad
\frac{\alpha_2 - y}{\alpha_1 - \alpha_2}=\frac{\alpha'_2 - y'}{\alpha'_1 - \alpha'_2}
$$
and
$$\frac{1 - x\bar\alpha_2}{x(\alpha_1 - \alpha_2)}= \frac{1 - x'\bar\alpha'_2}{x'(\alpha'_1 - \alpha'_2)},\quad
\frac{1 - y\bar \alpha_2}{y(\alpha_1 - \alpha_2)}= \frac{1 - y'\bar \alpha'_2}{y'(\alpha'_1 - \alpha'_2)}.$$ 
Moreover, $\frac{\alpha_1 - \alpha_2}{\bar\alpha_1 - \
\bar \alpha_2}= \frac{\alpha'_1 - \alpha'_2}{\bar\alpha'_1 - \bar\alpha'_2}$ and thus
$$\frac{\frac{1}{\bar x} - \alpha_2}{\alpha_1 - \alpha_2} = \frac{\frac{1}{\bar x'} - \alpha'_2}{\alpha'_1 - \alpha'_2},
\quad
\frac{\frac{1}{\bar y} - \alpha_2}{\alpha_1 - \alpha_2} = \frac{\frac{1}{\bar y'} - \alpha'_2}{\alpha'_1 - \alpha'_2}.$$
Subtracting and dividing proper equalities listed above we get
$$\frac{x-y}{\frac 1x - \frac 1y} = \frac{x' - y'}{\frac{1}{x'} - \frac{1}{y'}}, \quad \frac{x-\frac{1}{\bar y}}{y -\frac{1}{\bar x}} = \frac{x'-\frac{1}{\bar y'}}{y' -\frac{1}{\bar x'}}\text{ and } \frac{x-\frac{1}{\bar x}}{y -\frac{1}{\bar y}} = \frac{x'-\frac{1}{\bar x'}}{y' -\frac{1}{\bar y'}}.$$
The first relation immediately implies that $xy=x'y'$. From the second one we deduce that $|x' y|=|x y'|$, and consequently $|x'|=|x|$ and $|y'|=|y|$. Putting it to the last equality we find that $xy' = x'y$. These relations imply that $x = x'$, $y= y'$ or $x=-x'$, $y=-y'$. Making use of the equalities $x m_{\alpha_j}(x) = x' m_{\alpha_j'}(x')$, $y m_{\alpha_j}(y) = y' m_{\alpha_j'}(y')$, $j=1,\ldots, n$, we infer that either $(x,y,\alpha)=(x',y',\alpha')$ or $(x,y,\alpha) = (-x', -y', -\alpha')$.

Both these cases imply that $m_{t\cdot \alpha}(x) = m_{t'\cdot  \alpha}(x)$, and $m_{t\cdot \alpha}(y) = m_{t'\cdot  \alpha}(y)$. Therefore $t=t'$.
\end{proof}

Thus the claim follows and, according to the invariance of domain theorem, we have shown that $\Phi(X)$ is open in $\Omega$.
\end{proof}

\begin{lemma}\label{lem:phixclosed}
$\Phi(X)$ is closed in $\Omega$.
\end{lemma}

\begin{proof}Take a sequence $(z_\nu,w_\nu ,\xi_\nu)$ in $\Phi(X)$ converging to $(z',w',\xi')$ in $\Omega$. Let $(x_\nu, y_\nu, \alpha_\nu, t_\nu,\omega_\nu)\in X$ be such that $\Phi(x_\nu, y_\nu, \alpha_\nu, t_\nu,\omega_\nu) =(z_\nu,w_\nu,\xi_\nu).$ Put $\sigma_\nu:=x_\nu m_{t_\nu \cdot \alpha_\nu}(x_\nu)$ and $\tau_\nu:=y_\nu m_{t_\nu \cdot \alpha_\nu}(y_\nu).$ Clearly, $t_{z_\nu, w_\nu, (\sigma_\nu, \tau_\nu)}=1$. Passing to a subsequence we may assume that $(x_\nu, y_\nu, \alpha_\nu, t_\nu,\omega_\nu)$ is convergent in $\overline X$. Denote its limit by $(x',y',\alpha', t',\omega')$. We may also assume that $\sigma_\nu \to \sigma'$ and $\tau_\nu \to \tau'$. By the continuity of $t$ we find that $t_{z',w',(\sigma', \tau')}=1$ and thus the $3$-point Pick problem $\DD^n\to \DD$: 
\begin{equation}\label{eq:sigma'pick} 
\begin{cases}0\mapsto 0\\ z'\mapsto \sigma'\\ w'\mapsto \tau'
\end{cases}
\end{equation}
is extremal. 

Our aim is to show that $(x',y',\alpha', t',\omega')$ lies in $X$. Assume the contrary.

The simplest case is when both $x'$ and $y'$ are in $\DD$. Then trivially, $x'\neq y'$. If $t'\in b\Delta$, then the extremal problem corresponding to $(z', w', \xi')$ is $n-1$-dimensional. If $\alpha_j'\in \TT$ for some $j$, then there is a $2$-extremal passing through the nodes, whence \eqref{eq:sigma'pick} is 1-dimensional. Moreover if $\alpha_1=\alpha_2$, when $n=2$ (respectively, if $\alpha'_1$, $\alpha'_2$, $\alpha'_3$ are co-linear, when $n=3$), the problem is $n-1$-dimensional.

Assume that $x',y'\in\TT$. Note that then $x'=y'=\alpha_j'$ for $j=1, \ldots, n$. Moreover, $(\omega_{j-1} m_{(\alpha_j)_\nu}(x_\nu), \omega_{j-1} m_{(\alpha_j)_\nu}(y_\nu))$ is convergent to $(z_j', w_j'),$ $j=1,\ldots, n$, (here we put $\omega_0=1$) and $(m_{t_\nu\cdot \alpha_\nu} (x_\nu), m_{t_\nu\cdot \alpha_\nu} (y_\nu))$ converges to  $(\sigma', \tau')$. Therefore $\rho(x',y')= \rho(\sigma', \tau') =\rho(z_j', w_j')$, $j=1,\ldots, n$. Hence, problem~\eqref{eq:sigma'pick} is degenerate and $(z',w',\xi')$ does not lie in $\Omega$; a contradiction.

Now consider the case when, up to a permutation of $x'$ and $y'$, $x'\in \TT$ and $y'\in \DD$. Then $\alpha_j'\in\TT$ for all $j=1,\ldots, n.$ Using these relations we infer that $|w_j'|=|y'|$, $j=1,\ldots, n$, and $|\tau'|=|y'|$. Therefore, problem~\eqref{eq:sigma'pick} is degenerate, as well. This again gives a contradiction.
\end{proof}

\section{Proof for the bidisc}

In this section we shall present:
\begin{proof}[Proof of Lemma~\ref{main:bidisc}]
It follows from Lemma~\ref{lem:Omega(D2)domain} that $\Omega(\DD^2)$ is connected. Moreover, $\Phi(X(\DD^2))$ is an open-closed subset of $\Omega(\DD^2)$, according to Lemmas~\ref{lem:phixopen} and \ref{lem:phixclosed}. Therefore we get that 
$$\Phi(X(\DD^2))=\Omega(\DD^2).$$
From this equality and Remark~\ref{rem:degreeb} one can deduce the result.
\end{proof}

\section{Proof for the tridisc}\label{sec:tri}
The aim of this section if to prove Lemma~\ref{main:tridisc}. Our approach is quite technical and the proof is preceded by few preparatory results. Several times we shall make use of already proven Lemma~\ref{main:bidisc}.

Let us denote $$\mathbb X:=(\dD_3 \times \PP_1)\setminus \mathcal A_d.$$
Recall that $\Omega(\DD^3)= \XX\setminus \overline{U_1 \cup U_2}$.

For a subdomain $W$ of $\dD_3\times \PP_1$ by a topological boundary of $W$ with respect to $\XX$ we understand $\XX\cap \partial W$.

\begin{lemma}\label{lem:prep1} There is a closed set $\mathcal U$ such that the topological boundary of $U_1$ with respect to $\XX$ is contained in the union $\mathcal U\cup (\overline{U_1\cup U_2}^\circ)$ and the Hausdorff codimension of $\mathcal U$ is equal to $2$.
\end{lemma}

\begin{proof}[Proof]
Let $\mathcal I$ denote the following analytic set:
$$\mathcal I=\bigcup_{i\neq j}\{(z,w)\in \DD^3\times \DD^3:\ z_i=w_i\ \vee\ z_i=0\ \vee\ w_i=0,\ \vee \ z_i w_j = z_j w_i\}.$$ Take a point $(z,w)$ in the topological boundary of $V_1'$ with respect to $\DD^3\times \DD^3$ such that $(z,w)\not\in \mathcal I$. Thus, $|z_j|<|z_1|$, $|w_j|<|w_1|$ for $j=2,3,$ and the equality
\begin{equation}\label{eq:U1}
\rho\left(\frac{z_j}{z_1},\frac{w_j}{w_1} \right) = \rho(z_1,w_1)
\end{equation}
holds for some $j=2,3$ (for $(z,w)\in \partial V'_1\cap (\DD^3\times \DD^3)$ one of the equalities $|z_2|=|z_1|$, $|z_3|=|z_1|$ and the relation $(z,w)\in \mathcal I$ cannot hold simultaneously).

First, consider the situation when equality \eqref{eq:U1} is satisfied for two indexes, both $j=2$ and $j=3$. Let us denote 
\begin{multline}\nonumber \mathcal J:=\{(z,w)\in \DD^3\times \DD^3:\ |z_j|<|z_1|,\ |w_j|<|w_1|,\\ \rho\left(\frac{z_j}{z_1},\frac{w_j}{w_1} \right) = \rho(z_1,w_1),\ j=2,3\}.
\end{multline} 
Observe that the Hausdorff dimension of $\mathcal J$ is equal to $2n-2$ and that $\mathcal I \cup \mathcal J$ is closed. Thus $$\mathcal{U}:=\bigcup_{\sigma\in \Sigma_3} \sigma((\mathcal I\cup \mathcal J)\times \PP_1)$$ is a closed set such that whose Hausdorff codimension is equal to $2$.

What remains to do is to show that any point $(z,w,\xi)$, where $(z,w)$ satisfies the equality \eqref{eq:U1} with exactly one $j$, lies in $\overline{U_1\cup U_2}^\circ$.

Losing no generality suppose that 
\begin{equation}\label{eq:U_1} 
\rho\left(\frac{z_2}{z_1}, \frac{w_2}{w_1}\right)= \rho(z_1, w_1) \quad \text{and} \quad \rho\left(\frac{z_3}{z_1}, \frac{w_3}{w_1}\right)< \rho(z_1, w_1).
\end{equation}
The second relation in \eqref{eq:U_1} implies that $z_3 = \phi(z_1)$ and $w_3= \phi(w_1)$ for some holomorphic function $\phi:\DD\to \DD$ fixing the origin. Thus the extremal Pick interpolation problem $\DD^2\to \DD$ corresponding to $((z_1,z_2), (w_1, w_2), \xi)$ is non-degenerate (otherwise the problem corresponding to $(z,w,\xi)$ would be degenerate).
%(see Remark~\ref{rem:reduction})
Consequently, if $(z', w', \xi')$ is close to $(z,w,\xi)$, then, by Theorem~\ref{main:bidisc}, the function interpolating the extremal problem $\DD^2\to \DD$ corresponding to $((z_1', z_2'), (w_1', w_2'), \xi')$ is a left inverse to a $3$-complex geodesic passing through the nodes and being of the form stated there. Trivially, if $(z',w',\xi')$ is close enough to $(z,w,\xi)$, then there is a holomorphic function $\psi:\DD\to \DD $ such that $z_3' = z_1'\psi(z_1')$ and $w'_3= w_1'\psi(w_1')$. These properties mentioned above provide us with a 3-complex geodesic passing through the nodes of a three-extremal problem corresponding to $(z',w',\xi')$ and whose left inverse depends only on two first variables. Making use of this fact one can trivially check that $(z',w',\xi')\in \overline{U_1\cup U_2}$. This finishes the proof.
\end{proof}

\begin{lemma}\label{lem:b_1}
The topological boundary of $\Omega=\Omega(\DD^3)$ with respect to $\mathbb X$ may be divided onto two parts $b_1\Omega$ and $b_\infty \Omega$ such that: $b_1\Omega$ is a closed subset of $\mathbb X$ of the Hausdorff codimension $2$, and any point of $b_\infty \Omega$ is a smooth point of $\partial \Omega$.
\end{lemma}

\begin{remark}\label{rem:b_1}
Before we start to proof the lemma recall that we have already shown that the mapping 
$$\Phi:X(\DD^2) \to \Omega(\DD^2),$$ is surjective, $2-1$, and $\Phi(x,y,\alpha, t, \omega) = \Phi(-x,-y,-\alpha, t, \omega)$.

Thus from the equation $\Phi(x,y,\alpha, t, \omega)=(z,w,\xi)\in \Omega(\DD^2)$ we may derive $x,y, \alpha, t, \omega$ as real analytic functions of $(z,w,\xi)$, defined locally on $\Omega(\DD^2)$ (globally, up to a sign).
\end{remark}

\begin{proof}[Proof of Lemma~\ref{lem:b_1}]

We shall construct a relatively open subset $b_\infty \Omega$ of the boundary of $\Omega$ comprising points of smoothness of $\Omega$. To finish the proof it will be sufficient to show that the Hausdorff codimension of a set $b_1\Omega=(\partial \Omega \cap \XX) \setminus b_\infty \Omega$ is equal to $2$ (trivially $(\partial \Omega \cap \XX)\setminus b_\infty \Omega$ is closed in $\XX$). The set $b_1\Omega$ will be the union of sets $\mathcal U$, $\mathcal J_1$, $\mathcal J_2,$ and $\mathcal J_3$, where $\mathcal U$ appeared in Lemma~\ref{lem:prep1} and $\mathcal J_1,\mathcal J_2,\mathcal J_3$ are built in several steps in the sequel of the proof.

Take a point $(z',w',\xi')$ in the boundary of $\Omega=\mathbb X\setminus \overline{U_1\cup U_2}$ with respect to $\XX$. This means that $(z', w', \xi')\in \mathbb X \cap \partial (\overline{U_1\cup U_2}^\circ)$. It follows Lemma~\ref{lem:prep1} that $(z',w',\xi')\in \partial U_1$ implies $(z',w',\xi') \in \mathcal U$. Therefore it suffices to focus on the case when $(z',w',\xi')\in (\partial U_2)\setminus \overline{U_1}$. Then there is a permutation $\sigma$ such that  
\begin{equation}\label{eq:sigma_j}
(z',w',\xi') \in \partial\sigma (U_2').
\end{equation} 
Losing no generality suppose that \eqref{eq:sigma_j} holds for $\sigma=\id$.

Involving the case of the bidisc and using the notation from Remark~\ref{rem:b_1} one can see that $U_2'$ may be expressed in terms of $\Omega(\DD^2)$ as follows:
\begin{multline}\label{mult:U_2'} U_2' = \{(z,w,\xi):\ \eta(z,w,\xi)=(z_1,z_2,w_1, w_2 ,\xi)\in \Omega(\DD^2),\ |z_3|<|x (\eta)|,\\ |w_3|<|y(\eta)|,\ \rho\left(\frac{z_3}{x (\eta)}, \frac{w_3}{y (\eta)} \right) < \rho(x (\eta), y (\eta))\}.
\end{multline}
Denote $\eta'=\eta(z',w',\xi')= (z_1', z_2', w_1', w_2', \xi')$.

First consider the case when $\eta'\in \partial \Omega(\DD^2)$. Let $(z_\nu, w_\nu, \xi_\nu)_\nu$ be a sequence in $U_2'$ converging to $(z', w', \xi')$. Denote $\eta_\nu= \eta(z_\nu, w_\nu, \xi_\nu)$.
Take $(x_\nu, y_\nu, \alpha_\nu, t_\nu, \omega_\nu)$ in $X(\DD^2)$ such that $\Phi(x_\nu, y_\nu, \alpha_\nu, t_\nu, \omega_\nu) = \eta_\nu,$ $\nu \in \DD$.
Passing to a subsequence assume that $(x_\nu, y_\nu, \alpha_\nu, t_\nu, \omega_\nu)$ is convergent to $(x', y', \alpha', t', \omega')\in \partial X(\DD^2)$.  If $x',y'\in \DD$ and one of relations $x'=0$, $y'=0$, $x'=y'$, $t'\in \{0,1\}$ or $\alpha_1',\alpha_2'\in\DD$ are co-linear is satisfied, then the point $(z',w',\xi')$ lies in the following analytic set:
$$\mathcal J_1=\bigcup_i \{(z,w,\xi):\ z_i=0\ \vee\ w_i=0\ \vee\ z_i=w_i\ \vee\ \xi = [z_i: w_i]\}.
$$
Assume that $x',y',t',\alpha'$ do not satisfy these relation. We shall show that this contradicts the assumptions on $(z', w', \xi')$. Let us consider three cases.

a) $x',y'\in \DD$. Then $\alpha_j'\in \TT$ for some $j=1,2.$ Losing no generality assume that $j=1$. Then $z_1' = \alpha_1' x'$, $w_1' = \alpha_1' y'$, and $(z_2', w_2') = (z_1' \phi_1(z_1'), w_1' \phi_1 (w_1'))$ for some $\phi_1\in \mathcal O(\DD, \overline{\DD})$.

If $|z_3'|<|x'|$ and $|w_3'|<|y'|$, then $\rho(\frac{z_3'}{z_1'}, \frac{w_3'}{w_1'})\leq \rho(z_1', w_1')$. If these inequalities are not satisfied, then using the equalities $\rho(\frac{(z_3)_\nu}{x_\nu}, \frac{(w_3)_\nu}{y_\nu})\leq \rho(x_\nu, y_\nu)$ we infer that $z_3'=\omega x'$ and $w_3'=\omega y'$ for some unimodular constants $\omega$. Summing up,
there is $\phi_2\in \OO(\DD, \overline{\DD})$ such that $$(z_3', w_3') = (z_1' \phi_2(z_1'), w_1' \phi_2 (w_1')).$$ Thus, in this case $(z',w',\xi')$ lies in $\overline U_1$; a contradiction.

b) $x'\in \TT$ and $y'\in \DD$. Then $\alpha_1' = \alpha_2' = x'$, whence $|y'|=|w_1'|=|w_2'|$.
%, as $(z',w')\in \DD^2\times \DD^2$. %$w_0= (\alpha_1^0 y_0, \omega_0 \alpha_2^0 y_0)$.
Thus here the problem Pick problem $\DD^2\to \DD$ corresponding to the data $(\eta',\xi')$ is degenerate. More precisely, the sub-problem containing $0$ and $(w_1', w_2')$ is extremal (see the proof of Lemma~\ref{lem:phixclosed}). Clearly, $|w_3'|\leq |y'|$. %Since $(z',w',\xi')$ lies in the boundary of $U_2$, 
Thus the extremal Pick interpolation problem corresponding to $(z', w', \xi')$ is degenerate; a contradiction.

c) $x'\in \TT$ and $y'\in \TT$. Observe that $x'=y'=\alpha_1' = \alpha_2'$. Again, proceeding similarly as in the proof of Lemma~\ref{lem:phixclosed} we find that the extremal Pick problem $\DD^2\to \DD$ corresponding to the data $(\eta',\xi')$ is degenerate, that is the sub-problem composed of $z'$ and $w'$ is extremal. Then $\rho(\frac{(z_3)_\nu}{x_\nu}, \frac{(w_3)_\nu}{y_\nu})\leq \rho(x_\nu, y_\nu) = \rho(\frac{(z_j)_\nu}{x_\nu}, \frac{(w_j)_\nu}{x_\nu})$, $j=1,2$. Letting $\nu\to \infty$ we get $\rho(z'_3, w'_3)\leq \rho(z_j', w_j')$, $j=1,2$. Therefore the problem corresponding to $(z', w', \xi')$ is also degenerate, which again gives a contradiction.

\bigskip

We are left with the case $\eta'\in \Omega(\DD^2)$. If one of the equalities $|z_3'|<|x(\eta')|$, $|w_3'| < |y (\eta')|$ is not satisfied, there is a unimodular constant $\omega$ such that $z_3' = \omega x (\eta')$ and $w_3' = \omega y (\eta')$. 
Put 
\begin{multline}\mathcal J_2:=\bigcup_\sigma \sigma(\{(z,w,\xi):\ \eta=(z_1,z_2,w_1, w_2,\xi)\in \Omega(\DD^2),\\ z_3=\omega x(\eta), w_3=\omega y (\eta) \text{ for some } \omega\in \TT\}).
\end{multline} 
Note that the Hausdorff codimension of $\mathcal J_1$ is equal to $3$.

Suppose that $\eta'\in \Omega(\DD^2)$ and $|z_3'|<|x(\eta')|$, $|w_3'| < |y (\eta')|$. Then the boundary of $U_2'$ near $(z', w',\xi')$ is given by the equation $\rho(\frac{z_3}{x (\eta)}, \frac{w_3}{y (\eta)}) = \rho(x (\eta), y (\eta))$  and thus $U_2'$ is smooth here. 

We shall consider two possibilities depending on a number of permutation in $\Sigma_3$ condition~\eqref{eq:sigma_j} is satisfied for.

I) Suppose first that there is only one permutation for which \eqref{eq:sigma_j} is satisfied (we assume that it is the identity). This is a point of smoothness of $U_2$, as $U_2'$ is smooth there.
Denote the subset of such points in $\partial U_2'$ by $S$ and define \begin{equation}\label{eq:binfty}
b_\infty \Omega:=\bigcup_\sigma \sigma(S).
\end{equation}

II) Suppose that \eqref{eq:sigma_j} is satisfied for another permutation $\sigma$. First, observe that there are $\omega\in \TT$ and $\gamma\in \DD$ such that 
\begin{equation}\label{eq:z3}
(z'_3, w'_3) =  (\omega x(\eta') m_\gamma(x(\eta')), \omega y(\eta') m_\gamma(y(\eta'))).
\end{equation}
Let $\alpha(\eta')$ and $\omega(\eta')$ be as in Remark~\ref{rem:b_1}. Let  $\delta$ be a convex combination of $\alpha_1(\eta')$ and $\alpha_2(\eta')$ such that $\xi'=[x(\eta') m_\delta (x(\eta')): y(\eta') m_\delta (y(\eta'))]$.

Since $(z',w', \xi')$ is in $\partial \sigma(U_2')$, there is a function $F$ that interpolates the problem corresponding to $(z',w',\xi')$ independent on one of variables $z_1$ or $z_2$. Multiplying, if necessary, $F$ by a unimodular constant, we may assume that analytic discs $\lambda \mapsto F(\lambda m_{\alpha_1(\eta')} (\lambda), \omega(\eta') \lambda m_{\alpha_2(\eta')}(\lambda), \omega \lambda m_\gamma(\lambda))$ and $\lambda \mapsto \lambda m_{\delta}(\la)$ coincide at $0$, $x(\eta')$ and $y(\eta')$. Thus 
$$F(\lambda m_{\alpha_1(\eta')} (\lambda), \omega(\eta') \lambda m_{\alpha_2(\eta')}(\lambda), \omega \lambda m_\gamma(\lambda))= \lambda m_{\delta}(\lambda),\quad \la\in \DD.$$
It follows from Lemma~\ref{lem:convexcombination} that $\delta$ is a convex combination of $\gamma$ and $\alpha_1(\eta')$ or of $\gamma$ and $\alpha_2 (\eta')$. Both case imply that $\gamma$ lies on a real line passing through $\alpha_1(\eta')$, $\alpha_2(\eta')$. Therefore, making use of \eqref{eq:z3} we find $\XX\cap \partial U_2'\cap \partial \sigma(U_2')$ is contained in a set $\mathcal J_3$ whose Hausdorff codimension is equal to $2$.
\end{proof}

As an immediate consequence we get:
\begin{corollary}\label{cor:connected}
$\mathbb X\setminus b_1 \Omega$ is a domain.
\end{corollary}

\begin{lemma}\label{lem:eqbound}
Topological boundaries of $\Phi(X(\DD^3))$ and $\Omega(\DD^3)$ taken with respect to $\mathbb X\setminus b_1 \Omega$ do coincide.
\end{lemma}

\begin{proof}
Since $\Phi(X)$ is closed in $\Omega$ (see Lemma~\ref{lem:phixclosed}) we find that $\partial \Phi(X)$ is contained $\partial \Omega$. 

On the other hand the proof of Lemma~\ref{lem:b_1} (see formula~\eqref{eq:binfty}) contains a precise description of the topological boundary of $\Omega$ with respect to $\XX\setminus b_1\Omega $. Using this we get any point in $\partial\Omega\cap (\XX \setminus b_1\Omega)$ is of the form $(\psi(x), \psi(y),[xm_\delta(x): y m_\delta (y)])$, where $\psi_i(\lambda) = \omega_i m_{\alpha_i}(\lambda)$, $\alpha_i\in \DD$, $\omega_i$, $i=1,2,3$, and $\delta$ lies in the boundary of the simplex generated by $\alpha_1, \alpha_2, \alpha_3$.

Replacing $x$, $y$, $\alpha$, $\delta$ with $\omega x$, $\omega y$, $\omega \alpha$, $\omega\delta$, where $\omega$ is a proper unimodular constant, we may assume that $\omega_1=1$. Making use of the formula for $\Phi:X(\DD^3)\to \Omega(\DD^3)$ we get that $(\psi(x), \psi(y),[xm_\delta(x): y m_\delta (y)])$ lies in $\partial \Phi(X)$.
\end{proof}

\begin{lemma}\label{lem:Omega=Phi}
$\Phi(X(\DD^3))= \Omega(\DD^3)$.
\end{lemma}

\begin{proof}
Assume that the inclusion $\Phi(X)\subset \Omega$ is strict. Fix $x\in \Omega\setminus \Phi(X)$ and any $y\in \Phi(X)$. Thanks to Corollary~\ref{cor:connected} we may join $x$ and $y$ with a curve $\gamma:[0,1]\to \mathbb X\setminus b_1 \Omega$ starting at $x$. Let $t_0= \sup t$, where the supremum is taken over $t>0$ such that $\gamma([0,t))$ is contained in $\Omega$.

\begin{claim}\label{claim:gamma} There is no $s>0$ such that $\gamma([0,s])\subset \Omega$ and $\gamma(s)\in \Phi(X)$.
\end{claim}
\begin{proof}[Proof of claim]
If the claim were not true, $\gamma|_{[0,s]}$ would not intersect the boundary of $\Phi(X)$, by Lemma~\ref{lem:eqbound}. This implies that $\gamma([0,s])\subset \Phi(X)$; a contradiction. \end{proof}

We are coming back to the proof of the lemma. An immediate consequence of Claim~\ref{claim:gamma} is that $t_0\in (0,1)$. Of course $\gamma(t_0)$ lies in the boundary of $\Omega$. Since $\gamma(t_0)$ is a smooth point of $\Omega$ we may change the variables so that in a small neighborhood $U(t_0)$ of the point $\gamma(t_0)$ the boundary $\partial \Omega$ is a graph of a function and $U(t_0)\cap \Omega$ lies above it. It follows from Lemma~\ref{lem:eqbound} that $\Phi(X)\cap U(t_0)$ is a component of $U(t_0)\setminus \partial \Omega$ and since $\Phi(X)\subset \Omega$ we find that $\Phi(X)\cap U(t_0) = \Omega \cap U(t_0)$, providing that $U(t_0)$ is small enough. Consequently, $\gamma(t)\in \Phi(X)$ for $t<t_0$ sufficiently close to $t_0$, contradicting Claim~\ref{claim:gamma}.
\end{proof}

\begin{remark}\label{rem:connected}
Note that the set of triples $(\alpha_1, \alpha_2, \alpha_3)$, where $\alpha_j\in \DD^3$ are pairwise distinct and not co-linear, has two connected components. Clearly, if $\alpha$ belongs to one of them, $-\alpha$ lies in the second one. 

Consequently, $X(\DD^3)$ has two connected components. On the other hand, since $\Phi$ is even with respect to $(x,y, \alpha)$, we find that $\Omega(\DD^3)$ is connected.
\end{remark}

\begin{proof}[Proof of Lemma~\ref{main:tridisc}]
The result is a direct consequence of Lemma~\ref{lem:Omega=Phi}, the definition of $\Omega(\DD^3)$, the formula for $\Phi:X(\DD^3) \to \Omega(\DD^3)$, and Lemma~\ref{main:bidisc} together with Remark~\ref{rem:degreeb}.
\end{proof}

\section{Proof for $\DD^n$, $n\geq 4$}

\begin{remark}\label{rem:dense}
Suppose that the assertion of Theorem~\ref{main} is satisfied for the extremal problem corresponding to $(z_\nu,w_\nu,\xi_\nu)\subset \dD_n\times \PP_1$ for all $\nu \in \NN$. Suppose also $(z_\nu, w_\nu, \xi_\nu)$ converges to $(z_0,w_0,\xi_0)\in \dD_n \times \PP_1$. Then it is simple to see that either the extremal Pick interpolation problem corresponding to $(z_0,w_0,\xi_0)$ is degenerate or it is a left inverse to a $3$-complex geodesic passing through the nodes, i.e. it also satisfies the assertion of Theorem~\ref{main}.
\end{remark}

\begin{proof}[Proof of Lemma~\ref{main:poli}] 
The proof is quite technical. Some methods and arguments used below are similar to the ones involved in the proof of Theorem~\ref{main:tridisc}. In that case we shall sketch them rather then provide an accurate proof. 

It follows from Remark~\ref{rem:dense} that it is enough to show the assertion for a dense subset of $\dD_n\times \PP_1$.

For our convenience we shall remove a slightly bigger set from $\DD^n\times \DD^n \times \PP_1$. Let $\mathcal I=\{(z,w)\in \DD^n\times \DD^n:\ z_i=w_i \text{ or } z_i=0 \text{ or } w_i=0 \text{ for some }i\}$ and let $\dD_n'=\DD^n\times \DD^n\setminus \mathcal I$. Put $z_I=(z_{i_1},\ldots, z_{i_k})$, where $I=\{i_1,\ldots, i_k\},$  $1\leq i_1< \ldots < i_k\leq n$, $k\leq n$. Let $\mathcal B_d$ denote the set of points $(z,w,\xi)$ in $\dD_n' \times \PP_1$ such that for some $I$ the three-point Pick problem corresponding to the data $(z_I, w_I, \xi)$ is degenerate for some non-empty $I\subset \{1,\ldots, n\}$. Note that $\mathcal B_d$ is closed and that it is contained in the union of $\mathcal A_0$, $\mathcal A_d'$ and $\mathcal A_d''$. Thus the argument used in Lemma~\ref{lem:nondeg-connected} (with exactly the same proof) shows that $\XX':=(\dD_n'\times \PP_1)\setminus \mathcal B_d$ is a domain.

Let $b_1\Omega$ be as in Lemma~\ref{lem:b_1}. Put $b_n\Omega=\{(z,w,\xi)\in \XX':\ (z_I, w_I, \xi)\in b_1\Omega \text{ for some } I=\{i_1, i_2, i_3\}\}$. Since $b_n\Omega$ is closed in $\XX'$ and its Hausdorff codimension is equal to $2$ we infer that $\YY:=\XX'\setminus b_n\Omega$ is a domain.

We shall divide $\YY$ onto sets composed of points lying on $3$-extremals and such that the extremal Pick problem corresponding to them is strictly $j$ dimensional, $j=1,2,3$. These sets will be denoted respectively by $\Omega_j$, $j=1,2,3$. To get the assertion we shall show that the closure of $\Omega_1\cup \Omega_2 \cup \Omega_3$ is open in $\YY\setminus \mathcal A$ for some set $\mathcal A$ closed in $\YY$ and whose Hausdorff codimension is equal to $2$. The set $\mathcal A$ will be expressed as the union of $\mathcal A_i$, $i=1,2,3$, constructed in three steps below.

To simplify some expressions we shall introduce additional notation.
Let $$\psi_{\phi}(\lambda) = (\la, \la\phi_1(\la), \ldots, \la \phi_{n-1}(\la)),\quad \lambda\in \DD,$$ where $\phi\in \mathcal O(\DD, \DD^{n-1})$. Moreover, put $\omega_0=1$ and denote
$$\psi_{\alpha,\phi,\omega}(\la) = (\omega_0\la m_{\alpha_1}(\la),\ldots, \omega_{k-1}\la m_{\alpha_k}(\la), \la \phi_1(\la), \ldots, \la \phi_{n-k}(\la)),$$ where $\alpha\in \DD_k$, $\phi\in \mathcal O(\DD, \DD^{n-k})$, $\omega=(\omega_1,\ldots, \omega_{k-1})\in \TT^{k-1}$, $k=2,3$. Here $\DD_2$ (respectively $\DD_3$) denotes the subset of the bidisc (resp. tridisc) composed of $\alpha$ such that $\alpha_1\neq\alpha_2$ (resp. $\alpha_1$, $\alpha_2$, $\alpha_3$ are not co-linear). Let us define
$$
D_1=\{(\psi_{\phi}(x), \psi_{\phi}(y), \xi):\ x,y\in \DD\setminus \{0\},\ x\neq y,\ \phi\in \OO(\DD, \DD^{n-1})\}
$$
and
\begin{multline}
\nonumber
D_k = \{(\psi_{\alpha,\phi, \omega}(x), \psi_{\alpha, \phi, \omega}(y), [x m_{t\cdot \alpha}(x): y m_{t\cdot \alpha}(y)]):\ \alpha\in \DD_k,\\ x,y\in \DD\setminus \{0\},\ x\neq y,\ \omega\in \TT^{k-1}, \phi\in \mathcal O(\DD, \DD^{n-k}), t\in \Delta_k^\circ \},\quad k=2,3.
\end{multline}
Finally, we put $$\Omega_k=\bigcup_\sigma \sigma(D_k),\quad k=1,2,3,$$
where the union is taken over $n$-fold permutations.

To deal with the boundary of $\overline{\Omega_1 \cup \Omega_2 \cup \Omega_3}^\circ$ we shall focus on boundaries of $D_k$, $k=1,2,3$.

1) $D_1$ may be expressed as 
\begin{multline}\nonumber \{(z,w,\xi)\in \dD_n\times \PP_1:\ |z_j|<|z_1|,\ |w_j|<|w_1|,\\ \rho(z_j/z_1, w_j/w_1)< \rho(z_1,w_1),\quad j=2,\ldots,n \}.
\end{multline} 
In $\partial D_1$ we may distinguish special sets $$\{(z,w,\xi):\ z_j = \omega z_1, w_j = \omega w_1 \text{ for some } j=2,\ldots,n, \text{ and }\omega\in \TT\}$$ and 
\begin{multline}\nonumber \{(z,w,\xi):\ |z_j|<|z_1|, |w_j|<|w_1|,\quad j=2,\ldots,n, \text{ and}\\ \rho(z_k/z_1, w_k/w_1)= \rho(z_1, w_1) \text{ for two different indexes } k\}.
\end{multline} 
Denote their union by $\mathcal A_1'$ and put $\mathcal A_1=\bigcup_\sigma \sigma(\mathcal A_1')$, where the union is taken over $n$-fold permutations. Note that $\mathcal A_1$ is closed in $\YY$ and the Hausdorff codimension of $\mathcal A_1$ is equal to $2$.

Take a point $(z',w',\xi')$ in the boundary of $D_1$ with respect to $\YY$ omitting $\mathcal A_1$. We aim at showing that this points lies in $\overline{\Omega_1\cup \Omega_2 \cup \Omega_3}^\circ.$ 

There is $j$ such that $\rho(z_j'/z_1', w'_j/w'_1)= \rho(z'_1,w'_1)$ and 
\begin{equation}\label{eq:p} \rho\left(\frac{z'_k}{z'_1}, \frac{w'_k}{w'_1}\right) < \rho(z'_1,w'_1)
\end{equation}
for other indexes $k\geq 2$, $k\neq j$. Losing no generality we may assume that $j=2$. If $(z,w,\xi)$ is close enough to $(z',w',\xi')$, relations \eqref{eq:p} remain true. Making use of these inequalities and applying Lemma~\ref{main:bidisc} we find that $(z,w,\xi)\in \overline{D_1\cup D_2}^\circ$. 

2) We shall focus on $\partial D_2 \setminus \overline{\Omega_1}$. Note that $D_2$ may be expressed in terms of the inverse of $\Phi:X(\DD^2) \to \Omega(\DD^2)$ (see Remark~\ref{rem:b_1}) in the following way:
\begin{multline}\nonumber
D_2=\{(z,w,\xi):\ \eta=\eta(z,w,\xi)=((z_1,z_2), (w_1,w_2),\xi)\in \Omega(\DD^2),\ |z_j|<|x(\eta)|,\\ |w_j|<|y(\eta)|,\ \rho\left(\frac{z_j}{x(\eta)}, \frac{w_j}{y(\eta)} \right) < \rho(x(\eta), y(\eta)),\ j=3,\ldots,n\}.
\end{multline}

There is a closed set $\mathcal A_2$ in $\YY$ of Hausdorff codimension $2$ such that any point $(z,w,\xi)$ in $\partial D_2 \cap \YY$ satisfying
\begin{itemize}
\item $\eta\in \partial \Omega(\DD^2)$, or
\item $\eta\in \Omega(\DD^2)$ and $|z_j|=|x(\eta)|$ for some $j$, or
\item $\eta\in \Omega(\DD^2)$ and the equalities $\rho(\frac{z_j}{x(\eta)}, \frac{w_j}{y(\eta)}) = \rho(x(\eta), y(\eta))$ are satisfied for at least two distinct indexes $j\geq 3$,
\end{itemize}
lies either in $\overline{\Omega_1}$ or in $\mathcal A_2$.
Again, we may modify it so that $\sigma(\mathcal A_2) = \mathcal A_2$ for any $n$-fold permutation $\sigma$.

Take a point $(z',w',\xi')$ in $\partial D_2\cap \XX$ omitting $\mathcal A_2\cup \overline{\Omega_1}$. Then, up to a permutation of components, the following relations hold: $\eta'=((z_1', z_2'),(w_1', w_2'), \xi') \in \Omega(\DD^2)$, 
$|z'_k|<|x(\eta')|$, $|w_k'|<|y(\eta')|$ for $k\geq 3$ and 
$$
\rho\left(\frac{z'_3}{x(\eta')}, \frac{w'_3}{y(\eta')}\right) = \rho(x(\eta'), y(\eta')), \quad 
\rho\left(\frac{z'_j}{x(\eta')}, \frac{w'_j}{y(\eta')}\right) < \rho(x(\eta'), y(\eta')) \text{ for } j\geq 4.
$$ 
In particular, $D_2$ is smooth in a neighborhood of $(z',w',\xi')$. 

Denote $x'=x(\eta')$, $y'=y(\eta'),$ $t'=(t_1(\eta'), t_2(\eta'), 0),$ and $\alpha_j'=\alpha_j(\eta')$, $j=1,2.$ Let $\alpha_3'$ and $\omega_3'$ be such that $(z_3', w_3')= (\omega_3' x' m_{\alpha_3'}(x'), \omega_3' y' m_{\alpha_3'}(y'))$.

Let $(z,w,\xi)\in \YY$ be close enough to $(z',w',\xi')$ and such that $(z,w,\xi)\not\in \overline D_2$. Since $(z,w,\xi) \not\in b_n\Omega$ we see that $\theta(z',w',\xi')=((z_1',z_2',z_3'), (w_1', w_2', w_3'),\xi')$ is a point of smoothness of $\Omega(\DD^3)$ and that $\theta(z,w,\xi)=((z_1, z_2, z_3), (w_1, w_2, w_3),\xi)\in \Omega(\DD^3)$ (see Lemma~\ref{lem:b_1} for details). In particular, there is $(x,y,\omega, \alpha, t)\in X(\DD^3)$ such that 
$$
\Phi(x,y, \alpha,\omega, t) = \theta(z,w,\xi).
$$
\begin{claim} Up to a sign, $x$ is close to $x(\eta')$ providing that $(z,w,\xi)\not\in \overline{D_2}$ is close enough to $(z',w',\xi')$.
\end{claim}
\begin{proof}[Proof of claim] Assume a contrary. Then there is a sequence $(x_\nu, y_\nu, \alpha_\nu, \omega_\nu, t_\nu)_\nu \subset X(\DD^3)$ such that 
$\Phi(x_\nu, y_\nu, \alpha_\nu, \omega_\nu, t_\nu)\to \theta(z', w', \xi')$ and $x_\nu$ is far away from $x'$ and $-x'$. We may assume that this sequence is convergent in $\overline{X(\DD^3)}$. Let us denote its limit point by $(x^0,y^0,\alpha^0, \omega^0, t^0)$. Note that $\alpha^0_j\not\in \TT$, as otherwise the interpolation problem corresponding to $((z_1',z_2',z_3'), (w_1', w_2', w_3'), \xi')$ would be degenerate and consequently $(z',w',\xi')\in \mathcal B_d$. Similarly, $x^0\neq y^0$, $x^0, y^0\in \DD$ and $\alpha_1^0\neq \alpha_2^0$. In particular, $\Phi$ extends analytically in a neighborhood of $(x^0, y^0, \alpha^0, \omega^0, t^0)$ and 
\begin{equation}\label{eq:Phikoniec}
\Phi(x^0,y^0, \alpha^0, \omega^0, t^0)= \theta(z',w',\xi') = \Phi(x', y', \alpha', \omega', t').
\end{equation}
Let $F$ be a function depending on the first two variables interpolating the problem corresponding to $((z_1', z_2', z_3'), (w_1', w_2', w_3'), \xi')$ (recall that $t'_3=0$). From \eqref{eq:Phikoniec} we deduce that $F(\la m_{\alpha_1^0}(\la), \la m_{\alpha_2^0}(\la)) = \gamma \la m_{t^0\cdot \alpha^0}(\la)$, $\la\in \DD$ for some unimodular constant $\gamma$. Lemma~\ref{lem:convexcombination} implies that $t^0\cdot \alpha^0= s^0\cdot (\alpha_1^0, \alpha_2^0)$ for some $s^0\in \Delta_2^\circ$. We may rewrite~\eqref{eq:Phikoniec} in terms of $\Phi:X(\DD^2)\to \Omega(\DD^2)$ obtaining:
$$\Phi(x^0, y^0, (\alpha^0_1, \alpha_2^0), \omega_1^0, s^0) = \Phi(x', y', (\alpha_1', \alpha_2'), \omega_1', (t_1', t_2')).$$ It follows from Claim~\ref{claim} that either $x^0=x'$ or $x^0=-x'$. This gives a contradiction.
\end{proof}
It follows from the claim that the inequalities $\rho(z_j/x, w_j/y)< \rho(x,y)$ hold for every $j\geq 4$ if $(z,w,\xi)\not\in \overline D_2$ is sufficiently close to $(z',w', \xi')$.
Consequently, $(z,w,\xi)\in D_3$, whence $(z',w',\xi') \in \overline{D_2\cup D_3}^\circ$.

3) We are left with the boundary of $D_3$. Here we may express $D_3$ in terms of the inverse to $\Phi:X(\DD^3)\to \Omega(\DD^3)$ as follows:
\begin{multline}\nonumber
D_3=\{(z,w,\xi):\ \theta=((z_1, z_2, z_3), (w_1, w_2, w_3), \xi)\in \Omega(\DD^3),\\ \rho\left(\frac{z_j}{x(\theta)}, \frac{w_j}{y(\theta) }\right) < \rho(x(\theta), y(\theta)),\ j\geq 4\}.
\end{multline}
Let $(z',w',\xi')\in \YY\cap \partial D_3$. Similarly as before it is elementary to see that either $(z',w',\xi') \in \overline{D_1 \cup D_2}\cup \mathcal A_3$ for some closed set of Hausdorff codimension equal to $2$, or $\theta'=\theta(z',w',\xi)\in \Omega(\DD^3)$ and the equality
\begin{equation}
\label{eq:lst}
\rho\left(\frac{z'_j}{x(\theta')}, \frac{w'_j}{y(\theta') }\right) = \rho(x(\theta'), y(\theta'))
\end{equation}
is satisfied for exactly one $j$. Again we may replace, if necessary, $\mathcal A_3$ with $\bigcup_\sigma \sigma(\mathcal A_3)$.

If the second possibility mentioned above holds, that is when $\theta'\in \Omega(\DD^3)$ and equality \eqref{eq:lst} is satisfied, the boundary of $D_3$ is smooth in a neighborhood of $(z',w',\xi')$. 

Let $\omega\in \TT$ and $\alpha_4 \in \DD$ be such that 
$$
(z_j', w_j') = (\omega x(\theta') m_{\alpha_4}  (x(\theta')), \omega y(\theta') m_{\alpha_4}  (y(\theta'))).
$$

Clearly, $t(\theta')\cdot \alpha(\theta')$ lies in the open simplex generated by $\alpha_1(\theta'), \alpha_2(\theta'), \alpha_3(\theta')$. However, there is another simplex, say $\Delta_1$, containing $t(\theta')\cdot \alpha(\theta')$ that is generated by some three from four points $\alpha_1(\theta'),$ $\alpha_2(\theta'),$ $\alpha_3(\theta'),$ $\alpha_4$.

Thus, for some permutation $\sigma$ that switches $j$-th element with the 4-th one for some $j=1,2,3,$ the point $(z', w', \xi')$ lies in $\sigma(\overline{D_3})$. If $t(\theta')\cdot \alpha(\theta')$ lies in boundary of simplex $\Delta_1$, then the point $(z',w',\xi')$ lies in $\overline{\sigma(D_1)\cup \sigma(D_2)}$. That case is covered by considerations conducted above by 1) and 2). Otherwise $(z',w',\xi')$ is a point of smoothness $\sigma(D_3)$.

Note that in a neighborhood of this point the boundaries of $D_3$ and $\sigma(D_3)$ coincide. Thus to finish the proof we need to show that $D_3$ and $\sigma(D_3)$ are not equal in a small neighborhood of $(z',w',\xi')$. From this we will find immediately that $(z',w', \xi') \in \overline{D_3\cup \sigma(D_3)}^\circ.$

To show that $D_3$ and $\sigma(D_3)$ are not equal we shall prove more, namely that $D_3\cap \sigma(D_3)=\emptyset$. Assume the contrary, and take $(z,w,\xi)\in D_3 \cap \sigma(D_3)$. Then $(z,w,\xi) = (\psi_{\alpha, \phi, \omega}(x), \psi_{\alpha, \phi, \omega}(y))$ for proper $\alpha, \phi, \omega$ and $x,y$. On the other hand, since $(z,w,\xi)$ lies in $\sigma(D_3)$, there exists a function $F$ interpolating the extremal problem corresponding to $(z,w,\xi)$ which is in independent on $j$-th variable, where $j=1,2,3,$ is such that the permutation $\sigma$ switches $j$-th element with $4$.

Clearly $F\circ \psi_{\alpha,\phi, \omega}$ is a Blaschke product of degree $2$, so according to Remark~\ref{rem:degreeb}, $F$ is independent of the $4$-th variable, which means that the extremal problem corresponding to $(z,w,\xi)$ is $2$-dimensional. But this contradicts the fact that $(z,w,\xi)\in D_3.$

\end{proof}

\section{Relations with Geometric Function Theory problems}\label{sec:GTFP}

The results announced here have some interesting relations with Geometric Function Theory problems. We outline them below. All necessary definitions of objects we are dealing with and that are skipped here may be found in \cite{Jar-Pfl 2013}.

The most direct application is the Coman conjecture for the polydisc stating that the Green and Lempert functions with two poles of equal weights are equal for $\DD^n$. To state it we need some definitions:
Let $D$ ba a domain in $\CC^n$ and $\nu:D\to \mathbb R_+$ be a $\nu$-admissible (i.e. its support $\supp \nu$ is finite). Define $$d(\psi)=\sum_{w\in \supp \nu} \inf \{\nu(w) \log|\zeta|:\ \zeta\in \psi^{-1}(w)\}$$ and $$l_D(z, \nu) = \inf\{ d(\psi):\ \psi(0)=z,\ \psi\ \text{is a $\mu$-admissible disk}\}.$$ A function $l_D$ is called the \emph{generalized Lempert function} Moreover, let $$c_D(z, \nu)= \sup\{ \log|f(z)|:\ f\in \mathcal O(D, \mathbb D),\ \nu_{\log |f|}\geq \nu\}.$$ We call $c_D$ the \emph{generalized Carath\'eodory function}.

For a plurisubharmonic function $u$ in a neighborhood of a point $z_0$ in $\mathbb C^n$, the \emph{Lelong number} of $u$ at $z_0$, denoted throughout the paper by $\nu_u(z_0)$, is defined as the supremum of all $\nu$ such that $u(z)\leq \nu \log||z-z_0|| + O(1)$ for $z$ sufficiently close to $z_0$.

Recall the \emph{pluricomplex Green function with several poles} determined by $\nu$, where $\nu:D\to \mathbb R_+$ is admissible, $g_{D}(z,\nu)$ is given by
$$g_D (z,\nu)= \sup\{u<0:\ u\in \mathcal{PSH}_{-}(D),\ \nu_u\geq \nu\}.$$ The inequalities $c_D\leq g_D\leq l_D$ are trivial.

Coman conjectured in \cite{Com 2000} that the equality $g_g\equiv l_D$ holds. The equality was proved to be true for the Euclidean ball in $\CC^n$ and the bidisc in the case of two points with equal weight, i.e. if $\nu=\delta_p + \delta_q$. Nevertheless, it turned out that in general the conjecture is not true.

We shall need the following:
\begin{lemma} Let $\nu= \sum_{j=1}^N \delta_{p_j}$, $p_j\in D$. Then the equality $$c_D(z, \nu) = l_D(z,\nu)$$ holds if and only if a function interpolating the extremal problem
$$\begin{cases}
z\mapsto \sigma\\
p_1\mapsto \tau\\
\ldots\\
p_N\mapsto \tau
\end{cases}$$
is a left inverse to a $3$-complex geodesic passing through the nodes $z,p_1,\ldots, p_N$.
\end{lemma}
The proof of this fact is quite trivial. However, the point is that the Coman conjecture does not hold either for $N=3$ and the weights equal in the bidisc or in the case of any bounded domain
(see \cite{Tho 2012} and \cite{Wieg}). As a consequence we get that Theorem~\ref{main} cannot be extended to the case of $N$ points, with $N\geq 4$.
This shows that the case of three points is very special. On the other hand we suspect that it is not a rare phenomenon in the sense that similar properties probably hold for a bigger class of domains.

\bigskip

As already mentioned, the results presented here are new even for the bidisc just to mention that their simple consequence is a well known and definitely non-trivial fact (see \cite{A-Y2001} and \cite{Jar-Pfl 2013}) that all holomorphically invariant functions with one variable lying on the royal variety coincide in the so called symmetrized bidisc, denoted here by $\GG_2$ (the equality on $\GG_2\times \GG_2$ may be also concluded, however of that fact requires more than a two-lines argument). This may stated as follows:
\begin{prop}
Let $D$ be a domain in $\CC^n$ and let $F:\DD^n\to D$ be a proper holomorphic mapping of multiplicity $2$. Let $z$ be any point lying in the locus set of $F$. Then $$c_D (z, \cdot)\equiv l_D(z, \cdot).$$
\end{prop}
The proof of this fact is a simple consequence of our results and is skipped here as it lies beyond the scope of the paper.

There is also another domain for which the equality of holomorphically invariant metrics is non-trivial -- the tetrablock which, in turn, may be expressed as an image of the classical Cartan domain of the second type in $\CC^{2\times 2}$ under a proper holomorphic mapping of multiplicity $2$ (see \cite{AbouWhiYoun} and \cite{Edi-Kos-Zwo} for details). Thus the properties of both special domains as well as behavior or $3$-extremals in the Euclidean unit ball (see \cite{Kos-Zwo 2014} and \cite{War}) and the fact the Coman conjecture remains true there suggest that a counterpart of Theorem~\ref{main} holds in a bigger class of domains containing among others classical Cartan domains (it is worth pointing out that the Coman conjecture was also proved for the unit ball in $\CC^n$ -- see \cite{Com 2000} and \cite{Edi-Zwo 1998}).

A positive answer to that problem, beyond being interesting on its own (as a solution of the three Pick problem is not known for such a class of domains), would probably provide us with non-trivial domains for which the assertion of Lempert's theorem is satisfied (see \cite{Lem 1981}).

\end{document}